\documentclass[12pt,twoside]{amsart}
\usepackage{geometry}
\usepackage{amssymb,amsmath,amsthm, amscd, enumerate, mathrsfs}
\usepackage{graphicx, hhline}
\usepackage[all]{xy}
\usepackage{tikz}

\usepackage{color}

\usepackage[dvipdfmx]{hyperref}
\usepackage{mathtools}
\mathtoolsset{showonlyrefs}

\title{On non-projective complete toric varieties}
\author{Osamu Fujino and Hiroshi Sato} 
\date{2025/7/12, version 0.27}
\address{Department of 
Mathematics, Graduate School of Science, 
Kyoto University, Kyoto 606-8502, Japan}
\email{fujino@math.kyoto-u.ac.jp}
\address{Department of Applied Mathematics, 
Faculty of Sciences, Fukuoka University, 
8-19-1, Nanakuma, Jonan-ku, Fukuoka 814-0180, Japan}
\email{hirosato@fukuoka-u.ac.jp}
\keywords{non-projective varieties, toric varieties, minimal model program, 
flips, flops, anti-flips}
\subjclass[2020]{Primary 14M25; Secondary 14E30}

\dedicatory{Dedicated to Professor Tadao Oda on 
the occasion of his eighty-fifth birthday}

\newcommand{\Supp}[0]{{\operatorname{Supp}}}
\newcommand{\Proj}[0]{{\operatorname{Proj}}}
\newcommand{\Spec}[0]{{\operatorname{Spec}}}

\newcommand{\Exc}[0]{{\operatorname{Exc}}}
\newcommand{\NE}[0]{{\operatorname{NE}}}
\newcommand{\Bs}[0]{{\operatorname{Bs}}}
\newcommand{\G}[0]{{\operatorname{G}}}

\def\<{\ensuremath{\langle}}
\def\>{\ensuremath{\rangle}}

\newtheorem{thm}{Theorem}[section]
\newtheorem{cor}[thm]{Corollary}
\newtheorem{lem}[thm]{Lemma}
\newtheorem{conj}[thm]{Conjecture}

\theoremstyle{definition}
\newtheorem{defn}[thm]{Definition}
\newtheorem*{ack}{Acknowledgments}
\newtheorem{prob}[thm]{Problem} 
\newtheorem{rem}[thm]{Remark}
\newtheorem{ex}[thm]{Example}

\begin{document}

\maketitle 

\begin{abstract}
For every complete toric variety, 
there exists a projective toric variety which is isomorphic 
to it in codimension one. In this paper, 
we show that every smooth non-projective complete toric threefold 
of Picard number at most five becomes projective after a finite succession 
of flops or anti-flips. 
\end{abstract}

\tableofcontents 

\section{Introduction}\label{p-sec1}

Throughout this paper, we will work over an algebraically 
closed field $k$ of arbitrary characteristic. 
To the best knowledge of the authors, the following 
statement is not stated in the standard literature. 
Hence we state it here for the sake of completeness. 

\begin{thm}\label{p-thm1.1}
Let $X$ be a complete toric variety. 
Then we can always construct a projective 
$\mathbb Q$-factorial toric variety $X'$ which is 
isomorphic to $X$ in codimension one. 
\end{thm}

Theorem \ref{p-thm1.1} was inspired by Koll\'ar's problem in 
\cite{kollar}. 

\begin{prob}[{\cite[5.2.2.~Problem]{kollar}}]\label{p-prob1.2} 
Let $X$ be a proper algebraic threefold 
(smooth or with mild singularities). Can one find 
a (possibly very singular) projective variety $X^+$ such that $X$ and $X^+$ 
are isomorphic in codimension one? 
This means that there are subsets $B\subset X$ and $B^+\subset X^+$ 
and an
isomorphism $X\setminus B\simeq X^+\setminus B^+$ 
such that $\dim B\leq 1$, $\dim B^+\leq 1$.
\end{prob}

More generally, in this paper, we prove the following relative statement. 
It is obvious that if we put $Y=\Spec\,k$ in 
Theorem \ref{p-thm1.3} 
then we can recover Theorem \ref{p-thm1.1} as a special case. 

\begin{thm}\label{p-thm1.3}
Let $f\colon X\to Y$ be a proper surjective toric morphism. 
Then we can construct a commutative diagram 
\begin{equation*}
\xymatrix{
X\ar[dr]_-{f}\ar@{-->}[rr]^\phi& & X'\ar[dl]^-{f'} \\ 
&Y&
}
\end{equation*}
such that 
\begin{itemize}
\item[(i)] $f'\colon X'\to Y$ is a projective 
surjective toric morphism, 
\item[(ii)] $X'$ is $\mathbb Q$-factorial, and 
\item[(iii)] $\phi$ is a toric birational map which is an 
isomorphism in codimension one. 
\end{itemize}
\end{thm}

From the toric Mori theoretic viewpoint (see \cite{fujino},  
\cite{fujino-sato}, \cite{fujino-equivariant}, and so on), 
Theorem \ref{p-thm1.3} is natural and is not difficult to 
prove. 
Hence it is natural to pose the following conjecture, which was inspired 
by \cite[5.2.~Projectivization with flip or flop]{kollar}. 

\begin{conj}\label{p-conj1.4}
Let $X$ be a complete $\mathbb Q$-factorial toric 
variety. Then there exists a finite sequence consisting of flips, 
flops, or anti-flips 
\begin{equation*}
X=:X_0\dashrightarrow X_1\dashrightarrow \cdots 
\dashrightarrow X_m=:X'
\end{equation*} 
such that $X'$ is a projective $\mathbb Q$-factorial 
toric variety. 
\end{conj}

For the precise definition of flips, flops, and anti-flips, 
see Definition \ref{p-def2.2} below. The following 
observation seems to be more or 
less natural to the experts of the minimal model program.  

\medskip 

Let $X$ be a complete $\mathbb Q$-factorial toric variety. 
Then, by Theorem \ref{p-thm1.1}, we can take a projective 
$\mathbb Q$-factorial toric variety $X'$ which is isomorphic to 
$X$ in codimension one. We take an ample Cartier divisor $H'$ on $X'$. 
Let $H$ be the strict transform of $H'$ on $X$. 
Then $H$ is movable since $X$ is isomorphic to $X'$ in codimension one. 
Suppose that we can run the $H$-minimal model program. Then 
we have a finite sequence of flips 
\begin{equation*}
X=:X_0\dashrightarrow X_1\dashrightarrow \cdots 
\dashrightarrow X_m=:X^\dag
\end{equation*} 
such that $H^\dag$, which is the pushforward of $H$ on $X^\dag$, is nef and 
big. Note that every nef $\mathbb Q$-Cartier divisor on 
a complete toric variety is semiample. 
Hence we have a birational morphism 
\begin{equation*}
\Phi_{|lH^\dag|}\colon X^\dag
\to X'\simeq \Proj\bigoplus _{m=0}^\infty H^0(X', \mathcal O_{X'}(mH'))
\end{equation*} for some sufficiently large and divisible positive integer $l$. 
Since $X'$ is $\mathbb Q$-factorial and $X^\dag$ 
is isomorphic to $X'$ in codimension one, 
$\Phi_{|lH^\dag|}\colon X^\dag\to X'$ is an isomorphism. 
Then $X^\dag$ is projective. 
In particular, Conjecture \ref{p-conj1.4} holds. 
Unfortunately, however, we do not know whether we can run 
the minimal model program for non-projective complete 
toric varieties or not. 
Therefore, Conjecture \ref{p-conj1.4} is open. 
Now we have no good ideas to formulate the toric Mori 
theory for non-projective complete toric varieties.

\medskip 

One of the main results of this paper is as follows. 

\begin{thm}[see Theorem \ref{p-thm4.9}]\label{p-thm1.5}
Conjecture \ref{p-conj1.4} holds true for any smooth complete 
toric threefold of Picard number at most five. 
\end{thm}

Since every complete toric surface is projective, 
Theorem \ref{p-thm1.5} is a first non-trivial case of Conjecture 
\ref{p-conj1.4} (see Remark \ref{p-rem1.6} below). 

\begin{rem}\label{p-rem1.6}
It is well known that every smooth complete toric variety 
with Picard number at most three is always projective in 
any dimension (see \cite{ks}). 
Similarly, every $\mathbb Q$-factorial complete 
toric variety with Picard number at most two is projective 
(see \cite{rt}). Note that the fan $\Sigma$ in 
\cite[Example 1]{fujino-payne} defines a $\mathbb Q$-factorial 
non-projective complete toric threefold with 
Picard number three and no nontrivial 
nef line bundles. We also note that 
there is a smooth non-projective complete toric threefold 
with Picard number four (see Example \ref{p-ex4.1} below). 
\end{rem}

In Section \ref{p-sec4}, we will prove 
Theorem \ref{p-thm1.5} by using Oda's classification 
table of smooth complete toric threefolds with 
Picard number at most five in \cite{oda1} and \cite{oda2}. 
We think that the most interesting non-projective 
complete toric threefold with Picard number five 
is the one labeled as 
[8-13$''$] in \cite{oda1}. 
As a byproduct of the proof of Theorem 
\ref{p-thm1.5}, we have the following result. 

\begin{thm}[{see Corollary \ref{p-cor5.4}}]\label{p-thm1.7}
There exists a smooth non-projective complete toric threefold 
$X$ with Picard number five satisfying the following property. 
If $X\dashrightarrow X'$ is a toric birational map which is 
an isomorphism in codimension one to a projective toric variety $X'$, then 
$X'$ must be singular. 
\end{thm}

Theorem \ref{p-thm1.7} says that we can not make $X'$ smooth 
even when $X$ is smooth in Theorem \ref{p-thm1.3}. 
We know nothing about Conjecture \ref{p-conj1.4} in dimension 
at least four. 

\medskip 

The smooth complete toric varieties 
with Picard number at most five, which were classified by Tadao Oda (see \cite{oda1} 
and \cite{oda2}), are very special among all toric varieties. 
Therefore, it is not clear whether Theorem \ref{p-thm1.5} supports 
Conjecture \ref{p-conj1.4} or not. 
On the other hand, they give some interesting examples. 
It is well known that the toric variety labeled as [7-5] in \cite{oda1} 
is the simplest smooth non-projective complete toric variety. 
The one labeled as [8-12] in \cite{oda1} is the first example of 
smooth complete toric threefolds with Picard 
number five and no non-trivial nef line 
bundles (see \cite[Example 1]{fujino-payne}). 
In this paper, we show that the toric variety labeled 
as [8-13$''$] in \cite{oda1} gives an example that 
satisfies the property in Theorem \ref{p-thm1.7}. 
We hope that the description in 
Section \ref{p-sec4} will be useful for 
constructing some interesting examples or checking some conjectures. 

\medskip 

We look at the organization of this paper. 
In Section \ref{p-sec2}, we collect some definitions 
necessary for this paper. 
In Section \ref{p-sec3}, we give a detailed proof 
of Theorem \ref{p-thm1.3} based on the toric Mori theory for 
the sake of completeness. We do not need any combinatorial arguments in 
this section. 
The main part of this paper is Section \ref{p-sec4}, 
where we prove Theorem \ref{p-thm1.5}. 
More precisely, we 
check Theorem \ref{p-thm1.5} by using Oda's classification table 
of smooth complete toric threefolds with Picard number at most five. 
In the final section:~Section \ref{p-sec5}, we prove Theorem \ref{p-thm1.7}. 

\medskip 

The reader can find other interesting 
non-projective complete toric varieties 
(resp.~non-projective complete surfaces) in \cite{bonavero} and 
\cite{fujino-km-cone} (resp.~\cite[Section 12]{fujino-classC}). 

\begin{ack}\label{p-ack}
The first author was partially 
supported by JSPS KAKENHI Grant Numbers 
JP20H00111, JP21H04994, JP23K20787. 
The second author was partially
supported by JSPS KAKENHI Grant Number 
JP24K06679. 
The authors thank the referees 
for many useful comments and suggestions. 
\end{ack}

\section{Preliminaries}\label{p-sec2} 

In this section, we collect some definitions and 
properties of toric varieties necessary for this paper. 
For the basic definitions and results of toric geometry, 
see \cite{cls}, \cite{fulton}, \cite{oda1}, \cite{oda2}, and so on. 
For the basic definitions of the minimal 
model theory, see \cite{fujino-foundations}. 
The following lemma is very well known. 

\begin{lem}[$\mathbb Q$-factoriality]\label{p-lem2.1}
Let $X=X_\Sigma$ be a toric variety with its fan $\Sigma$. 
Then $X$ is {\em{$\mathbb Q$-factorial}} if and only if 
each cone $\sigma\in \Sigma$ is simplicial. 
\end{lem}

We need flips, flops, and anti-flips. 

\begin{defn}[Flips, flops, and anti-flips]\label{p-def2.2}
Let $\varphi\colon X\to W$ be a small projective toric 
birational morphism from a $\mathbb Q$-factorial 
toric variety $X$ such that $\rho (X/W)=1$. 
Let $D$ be a $\mathbb Q$-Cartier $\mathbb Q$-divisor on $X$. 
If $-D$ is $\varphi$-ample, 
then $\varphi\colon X\to W$ is called 
a {\em{$D$-flipping contraction}}. 
In this setting, we can always construct the 
following commutative diagram: 
\begin{equation*}
\xymatrix{
X \ar@{-->}[rr]^-\phi\ar[dr]_-\varphi && \ar[dl]^-{\varphi^+} X^+ \\ 
&W&
}
\end{equation*}
such that 
\begin{itemize}
\item[(i)] $\varphi^+\colon X^+\to W$ is a small 
projective toric birational morphism, 
\item[(ii)] $X^+$ is $\mathbb Q$-factorial, 
\item[(iii)] $\rho(X^+/W)=1$, and 
\item[(iv)] $D^+$ is $\varphi^+$-ample, 
where $D^+$ is the strict transform of $D$ on $X^+$. 
\end{itemize}
We usually say that $\phi\colon X\dashrightarrow X^+$ 
is a {\em{$D$-flip}}. 
Note that we sometimes simply say that $\phi\colon X\dashrightarrow 
X^+$ is a {\em{flip}} if there is no danger of confusion. 
Let $r$ be a positive integer such that $rD$ is Cartier. 
Then it is well known that $\varphi^+\colon X^+\to W$ is 
the natural projection 
\begin{equation*}
\pi\colon \Proj_W\bigoplus _{m=0}^\infty 
\varphi_*\mathcal O_X(mrD)\to W.
\end{equation*}  
If $-K_X$ (resp.~$K_X$) is $\varphi$-ample, 
then $\phi\colon X\dashrightarrow X^+$ is simply called 
a {\em{flip}} (resp.~an {\em{anti-flip}}). 
If $K_X$ is $\varphi$-numerically trivial, 
then $\phi\colon X\dashrightarrow X^+$ is called a {\em{flop}}. 
\end{defn}

For the combinatorial description of $\varphi\colon X\to W$ and $\varphi^+\colon 
X^+\to W$, see \cite{reid}, \cite{matsuki}, \cite{sato2}, and so on. 

\begin{rem}\label{p-rem2.3}
For simplicity, we assume that $X$ is a $\mathbb Q$-factorial 
projective toric variety. 
In Definition \ref{p-def2.2}, the relative Kleiman--Mori 
cone $\NE(X/W)$ is obviously a half line. However, $\NE(X/W)$ is 
not necessarily an extremal ray of $\NE(X)$. 
Note that $\NE(X/W)$ is an extremal ray of $\NE(X)$ 
if and only if $W$ is projective. 
\end{rem}

We make a small remark about smooth toric threefolds. Note that 
we are mainly interested in threefolds in this paper. 

\begin{rem}\label{p-rem2.4}
If $X$ is a smooth toric threefold, then 
$X$ has no flipping contractions (see, for example, 
\cite{fstu}). 
If $X$ is a smooth threefold and 
$\phi\colon X\dashrightarrow X^+$ is a flop in Definition \ref{p-def2.2}, 
then $X^+$ is also a smooth threefold. 
\end{rem}

The following notion of primitive collections and relations 
introduced by Batyrev \cite{batyrev1} 
is convenient to describe 
the combinatorial information of a simplicial complete fan. 
For more informations, please see 
\cite{batyrev1}, \cite{batyrev2} and \cite{sato1}. 

\begin{defn}[Primitive collections]\label{p-def2.5}
Let $X=X_\Sigma$ be a $\mathbb{Q}$-factorial complete toric variety 
with its fan $\Sigma$. 
Let $\G(\Sigma)$ be 
the set of all 
primitive generators for one-dimensional cones in $\Sigma$. 
Then, we call a non-empty subset $P\subset\G(\Sigma)$ 
a {\em primitive collection} if $P$ does not generate a cone in $\Sigma$, 
while any proper subset of $P$ generates a cone in $\Sigma$. 
\end{defn}

From all the primitive collections of a fan $\Sigma$, we can know the 
combinatorial information of $\Sigma$. Moreover, when $X$ is smooth, 
we can recover $\Sigma$ from the data of the following primitive relations. 

\begin{defn}[Primitive relations]\label{p-def2.6}
Let $X=X_\Sigma$ be a {\em smooth} complete toric variety. 
For a primitive collection $P=\{v_1,\ldots,v_r\}\subset\G(\Sigma)$, 
there exists the unique cone $\sigma\in\Sigma$ 
which contains $v_1+\cdots+v_r$ in its relative interior. 
Let $\sigma\cap\G(\Sigma)=\{w_1,\ldots,w_s\}$. 
Thus, we have the equality 
\[
v_1+\cdots+v_r=a_1w_1+\cdots+a_sw_s\quad (a_1,\ldots,a_s\in\mathbb{Z}_{>0}), 
\]
which we call the {\em primitive relation} for $P$. 
\end{defn}

The following lemma is also well known. 

\begin{lem}[{\cite[Corollary 2.14]{oda2}}]\label{p-lem2.7}
Let $X=X_\Sigma$ be a complete toric variety with its fan $\Sigma$. 
Then $X$ is projective if and only if 
there exists a strictly convex $\Sigma$-linear 
support function. 
\end{lem}

For the details of 
Lemma \ref{p-lem2.7}, 
see \cite{cls}, \cite{fulton}, \cite{oda1}, \cite{oda2}, and so on. 
In our setting, we have: 

\begin{lem}\label{p-lem2.8}
Let $X=X_\Sigma$ be a $\mathbb Q$-factorial 
complete toric variety and let $G(\Sigma)=\{v_1, \ldots, v_l\}$ be 
the set of all primitive generators for one-dimensional 
cones in $\Sigma$. Let $D=\sum _i d_i D_i$ be 
a torus invariant 
effective Cartier divisor 
on $X$, 
where $D_i$ is the torus-invariant prime divisor corresponding to 
$v_i$ for every $i$. 
Let $\Psi _D$ be the piecewise linear function associated to $D$, 
that is, $\Psi_D$ is the 
piecewise linear function defined by $\Psi_D(v_i)=d_i$ 
for every $i$. 
Then $D$ is ample if and only if $-\Psi _D$ is strictly convex. 
\end{lem} 

We will repeatedly use the following lemma to prove the non-projectivity 
of a given toric variety in Section \ref{p-sec4}. 
It is obvious by Lemma \ref{p-lem2.8}. 

\begin{lem}\label{p-lem2.9}
Let $X=X_\Sigma$ be a smooth complete toric variety,  
$\G(\Sigma)=\{v_1,\ldots,v_l\}$ and $D_1,\ldots,D_l$ the 
torus-invariant prime divisors on $X$ 
corresponding to $v_1,\ldots,v_l$, respectively. 
If an effective divisor $D=
\sum_{i}d_iD_i\ 
(d_i\ge 0)$ is ample, then 
for every primitive relation 
\[
v_{i_1}+\cdots+v_{i_r}=a_1v_{j_1}+\cdots+a_sv_{j_s}\quad (a_1,\ldots,a_s
\in \mathbb Z_{>0}), 
\]
the inequality 
\[
d_{i_1}+\cdots+d_{i_r}-(a_1d_{j_1}+\cdots+a_sd_{j_s})>0
\]
holds. 
\end{lem}

We will frequently use the following 
well-known lemma in Section \ref{p-sec4} to prove 
the projectivity of certain toric varieties. 

\begin{lem}\label{p-lem2.10}
Let $X = X_\Sigma$ be a $\mathbb Q$-factorial complete toric 
variety associated with a fan $\Sigma$. Let $\Sigma'$ be the 
star subdivision of $\Sigma$ along a ray generated by 
a vector $v$. Then, the associated toric morphism
\[
\varphi \colon X_{\Sigma'} \to X = X_\Sigma
\]
is a projective birational toric morphism. 
This morphism can be viewed as 
the toric contraction morphism that contracts the divisor $V(\langle v \rangle)$ 
on $X_{\Sigma'}$. Note that $\varphi \colon X_{\Sigma'} \to X_\Sigma$ is projective. 
Therefore, if $X_\Sigma$ is projective, then so is $X_{\Sigma'}$. 
\end{lem}

\section{Proof of Theorem \ref{p-thm1.3}}\label{p-sec3}

In this section, we give a detailed proof of Theorem \ref{p-thm1.3} 
based on the toric Mori theory for the reader's 
convenience (see, 
for example, \cite{fujino}, \cite{fujino-sato}, \cite{fujino-equivariant}, 
and so on). 
We start with an easy lemma. 
We prove it here for the sake of completeness. 

\begin{lem}\label{p-lem3.1}
Let $X$ be a toric variety. 
Then there exists a projective birational toric morphism 
$X'\to X$ from a smooth quasi-projective toric variety $X'$. 
\end{lem}

\begin{proof}[Proof of Lemma \ref{p-lem3.1}]
By Sumihiro's equivariant completion 
(see \cite[Theorem 3]{sumihiro} and \cite[p.17]{oda2}), 
we may assume that $X$ is complete. 
By Chow's lemma for toric varieties (see 
\cite[Proposition 2.17]{oda2}), 
we can construct a projective birational toric 
morphism $X''\to X$ from a projective toric 
variety $X''$. 
By the toric desingularization theorem (see, for example, 
\cite[Theorem 11.19]{cls}), 
we can construct a projective birational toric morphism 
$X'\to X''$ from a smooth toric variety $X'$. 
Then the induced map $X'\to X$ is a desired morphism. 
We finish the proof. 
\end{proof}

Let us prove Theorem \ref{p-thm1.3}. 

\begin{proof}[Proof of Theorem \ref{p-thm1.3}]
By taking a small projective $\mathbb Q$-factorialization 
(see \cite[Corollary 5.9]{fujino}), 
we may assume that $X$ is $\mathbb Q$-factorial. 
If $X$ is projective over $Y$, then there is nothing to prove. 
Hence, from now, we may assume that $X$ is not projective over $Y$. 
We take a projective birational toric morphism 
$p\colon V\to X$ from a smooth quasi-projective toric 
variety $V$ (see 
Lemma \ref{p-lem3.1}). 
Since $X$ is $\mathbb Q$-factorial, we can take 
an ample Cartier divisor $A$ on $V$ such that 
$H:=p_*A$ is Cartier. 
We put $D:=p^*H+E$, 
where $E$ is an effective Cartier divisor on $V$ with 
$\Supp E=\Exc(p)$. 
We note that 
\begin{equation*}
\bigoplus _{m=0}^\infty (f\circ p)_*\mathcal O_V(mD) 
\end{equation*} 
is a finitely generated $\mathcal O_Y$-algebra (see 
\cite[Corollary 5.8]{fujino-sato}) and that 
\begin{equation*}
(f\circ p)_*\mathcal O_V(mD)\simeq f_*\mathcal O_X(mH)
\end{equation*} 
holds for every non-negative integer $m$. We consider 
\begin{equation*}
Z:=\Proj _Y\bigoplus _{m=0}^\infty (f\circ p)_*\mathcal O_V(mD)
\simeq \Proj _Y\bigoplus _{m=0}^\infty f_*\mathcal O_X(mH).  
\end{equation*} 
Then $Z$ is a variety which is projective 
over $Y$ by construction. 
By running the $D$-minimal model program over $Y$ (see 
\cite[5.3]{fujino} and \cite[3.1]{fujino-sato}), 
we have a sequence of flips and 
divisorial contractions with respect to $D$ over $Y$ starting from $V$: 
\begin{equation*}
V=:V_0\dashrightarrow V_1\dashrightarrow \cdots 
\dashrightarrow V_l=:V'
\end{equation*} 
such that $D'$, which is the pushforward of $D$ on $V'$, 
is nef and big over $Y$. 
Note that $V_i$ is $\mathbb Q$-factorial for every $i$ because $V$ is. 
By the definition of $Z$, 
we have a birational contraction morphism 
$V'\to Z$ over $Y$. 
We note that $Z$ is a toric variety, since 
the image of any contraction morphism from a toric variety is again 
a toric variety (see, e.g., \cite[Proposition~4.6]{fujino-sato} and 
\cite[Proposition~2.7]{tanaka}). 
Let 
\begin{equation*}
\xymatrix{
&\ar[dl]_-\alpha W\ar[dr]^-\beta& \\ 
V\ar@{-->}[rr]&& V'
}
\end{equation*}
be a common resolution. 
Then by applying the negativity lemma 
(see, for example, \cite[Lemma 4.10]{fujino-sato}) to 
each step of the above minimal model program, 
we can write 
\begin{equation*}
\alpha^*D=\beta^*D'+G
\end{equation*} 
where $G$ is an effective $\beta$-exceptional 
$\mathbb Q$-divisor on $W$. 
Let $m$ be  a sufficiently large and divisible positive 
integer. 
Then 
\begin{equation*}
mG=\Bs |m(\beta^*D'+G)|= \Bs |m\alpha^*D| 
=\Bs| m\alpha^*(p^*H+E)|\supset m\alpha^*E
\end{equation*} 
holds. Hence $\alpha^*E$ is $\beta$-exceptional. 
This implies that every irreducible component of $E$ is 
contracted by $V\dashrightarrow V'$. 
Therefore, the induced birational map $\psi\colon X\dashrightarrow 
Z$ over $Y$ is a contraction, that is, $\psi^{-1}$ contracts no divisors. 
We take a small projective $\mathbb Q$-factorialization $Z'\to Z$ (see 
\cite[Corollary 5.9]{fujino}). 
Then the induced map $\psi'\colon X\dashrightarrow Z'$ over $Y$ is also a contraction. 
If $\psi'$ contracts some divisors, then we take a composite of 
blow-ups $X'\to Z'$ extracting those divisors. 
Thus the induced toric birational map 
$X\dashrightarrow X'$ over $Y$ is an isomorphism 
in codimension one and $X'$ is a $\mathbb Q$-factorial 
toric variety which is projective over $Y$ by construction. 
This is what we wanted. 
\end{proof}

Now Theorem \ref{p-thm1.1} is obvious. 

\begin{proof}[Proof of Theorem \ref{p-thm1.1}]
We put $Y=\Spec \, k$. 
Then Theorem \ref{p-thm1.1} is an obvious consequence of Theorem 
\ref{p-thm1.3}. 
\end{proof}

\section{Smooth complete toric threefolds with $\rho\leq 5$}\label{p-sec4}

In this section, we confirm Conjecture \ref{p-conj1.4} for smooth complete 
toric threefolds of Picard number at most five. 
If the Picard number of a smooth complete toric variety is less than four, 
then the variety is always projective (see Remark \ref{p-rem1.6}). 
So, first, we quickly review the famous smooth non-projective 
complete toric threefold of Picard number four 
described in \cite[Proposition 9.4]{oda1}. 

\begin{ex}\label{p-ex4.1}
Let $W:=X_\Sigma$ be the smooth complete toric threefold associated to 
the fan $\Sigma$ in $\mathbb{R}^3$ whose one-dimensional cones are 
generated by 
\[
v_1 = (1,0,0),  v_2 = (0,1,0), 
v_3 = (0,0,1),v_4 = (-1,-1,-1), 
\]
\[
v_5 = (-1,-1,0),  v_6 = (0,-1,-1), 
v_7 = (-1,0,-1),
\]
and the three-dimensional cones of $\Sigma$ are 
\[
\<v_1, v_2, v_3 \>,  \<v_1, v_2, v_7 \>,  \<v_1, v_3, v_6 \>, 
\<v_1, v_6, v_7 \>,  \<v_2, v_3, v_5 \>, 
\]
\[
 \<v_2, v_5, v_7 \>,  \<v_3, v_5, v_6 \>, 
\<v_4, v_5, v_6 \>,  \<v_4, v_5, v_7 \>,  \<v_4, v_6, v_7 \>.
\]
The configuration of cones of $\Sigma$ are as the following diagram: 
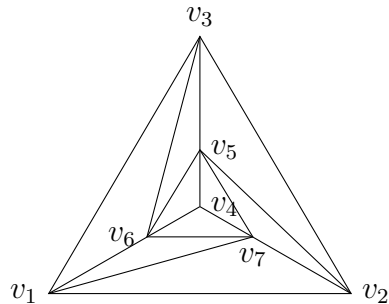
\begin{center}
\begin{tikzpicture}

\coordinate (D) at (4*0.5,3.8*0.5);
\coordinate (C) at (2.6*0.5,1.5*0.5);
\coordinate (B) at (5.4*0.5,1.5*0.5);
\coordinate (G) at (0,0);
\coordinate (E) at (4*0.5,6.8*0.5);
\coordinate (F) at (8*0.5,0);
\coordinate (A) at (2,2.3*0.5);

\draw (A)--(F);

\draw (A)--(G);

\draw (A)--(E);

\draw (B)--(G);
\draw (C)--(E);
\draw (D)--(F);

\draw (B)--(C)--(D)--cycle;
\draw (E)--(F)--(G)--cycle;

\draw (G) node[left] {$v_1$};
\draw (F) node[right] {$v_2$};
\draw (E) node[above] {$v_3$};
\draw (A) node[right] {$v_4$};
\draw (D) node[right] {$v_5$};
\draw (C) node[left] {$v_6$};
\draw (B) node[below] {$v_7$};

\end{tikzpicture}
\end{center}
We should remark that there is the extra three-dimensional cone 
$\langle v_1,v_2,v_3\rangle$. 
The variety $W$ is the smooth complete toric threefold 
of type [7-5] in \cite{oda1}. 

One can easily see there exist exactly three flopping curves 
\[
V(\<v_1,v_7\>),\ V(\<v_2,v_5\>)\mbox{ and }V(\<v_3,v_6\>), 
\]
that is, there are primitive relations 
\[
v_2+v_6=v_1+v_7,\ v_3+v_7=v_2+v_5\mbox{ and }v_1+v_5=v_3+v_6,
\]
correspondingly. This configuration of cones associated to the three 
flopping curves causes the non-projectivity of $W$. 
More precisely, the above primitive relations 
and Lemma \ref{p-lem2.9} imply that there is no 
effective ample Cartier divisor on $W$. 
Hence $W$ is non-projective. As remarked in \cite[p.71]{oda1}, 
after one of these three flops, $W$ becomes projective. 
In particular, Conjecture \ref{p-conj1.4} is true for $W$. 
We remark that among smooth complete toric threefolds of Picard number four, 
only $W$ is non-projective by the classification in \cite{oda1}. 
\end{ex}

Next, we consider the case of Picard number five. 
By the classification in \cite[Theorem 9.6]{oda1}, it is well known that 
there exist exactly eleven types 
\[
\mbox{
[8-2], 
[8-5$'$], 
[8-5$''$], 
[8-8], 
[8-10], 
[8-11], 
[8-12], 
[8-13$'$], 
[8-13$''$], 
[8-14$'$] and 
[8-14$''$]}
\]
of smooth complete toric threefolds of Picard number five 
which cannot be blown-down to smooth threefolds. Among them, 
[8-2], 
[8-10] and 
[8-11] 
are projective. 
Let us start with the proof of 
the projectivity of these varieties for reader's convenience. 
We emphasize that 
the projectivity of 
the smooth complete toric threefolds of type [8-11] 
was not announced in \cite{oda1}. 
In the following, we use the notation in \cite{fujino-payne} 
which is slightly different from the one in \cite{oda1}. 

\bigskip

\noindent
{\bf [8-2]}\quad 
Let $a\in\mathbb{Z}$ and $Z_{2}(a):=X_\Sigma$ 
the smooth complete toric threefold associated to 
the fan $\Sigma$ in $\mathbb{R}^3$ whose one-dimensional cones are 
generated by 
\[
v_1 = (1,0,0),  v_2 = (0,-2,-1), 
v_3 = (0,-1,0),v_4 = (0,0,1), 
\]
\[
v_5 = (0,1,a),  v_6 = (0,0,-1), 
v_7 = (0,-1,-1),  v_8 = (-1,-3,-2),
\]
and the three-dimensional cones of $\Sigma$ are 
\[
\<v_1, v_2, v_3 \>,  \<v_1, v_2, v_8 \>,  \<v_1, v_3, v_4 \>, 
\<v_1, v_4, v_5 \>,  \<v_1, v_5, v_6 \>, 
 \<v_1, v_6, v_7 \>,  
 \]
 \[
 \<v_1, v_7, v_8 \>, 
\<v_2, v_3, v_8 \>,  \<v_3, v_4, v_8 \>,  \<v_4, v_5, v_8 \>,  \<v_5, v_6, v_8\>,
\<v_6,v_7,v_8\>.
\]
The configuration of cones of $\Sigma$ are as the following diagram: 
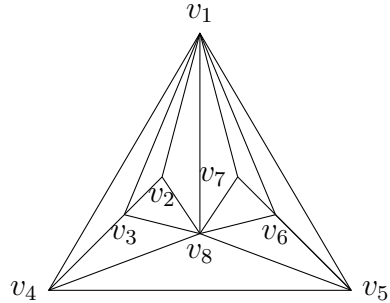
\begin{center}
\begin{tikzpicture}

\coordinate (H) at (2,1.5*0.5);
\coordinate (D) at (0,0);
\coordinate (A) at (4*0.5,6.8*0.5);
\coordinate (E) at (8*0.5,0);

\coordinate (C) at (1,1);
\coordinate (B) at (1.5,1.5);

\coordinate (G) at (4-1.5,1.5);
\coordinate (F) at (4-1,1);

\draw (H)--(A);
\draw (H)--(D);
\draw (H)--(E);

\draw (C)--(D);
\draw (C)--(A);
\draw (C)--(H);

\draw (B)--(C);
\draw (B)--(A);
\draw (B)--(H);

\draw (F)--(E);
\draw (F)--(H);
\draw (F)--(A);

\draw (G)--(A);
\draw (G)--(H);
\draw (G)--(E);

\draw (A)--(D)--(E)--cycle;

\draw (A) node[above] {$v_1$};
\draw (B) node[below] {$v_2$};
\draw (C) node[below] {$v_3$};
\draw (D) node[left] {$v_4$};
\draw (E) node[right] {$v_5$};
\draw (F) node[below] {$v_6$};

\draw (G) node[left] {$v_7$};
\draw (H) node[below] {$v_8$};

\end{tikzpicture}
\end{center}
There is the extra three-dimensional cone 
$\langle v_1,v_4,v_5\rangle$. 

We obtain the sequence 
\[
Z_2(a)\stackrel{\varphi_1}{\longrightarrow} 
Y_1\stackrel{\varphi_2}{\longrightarrow} Y_2\stackrel{\varphi_3}{\longrightarrow} Y_3
\]
of birational morphisms associated to the diagrams 
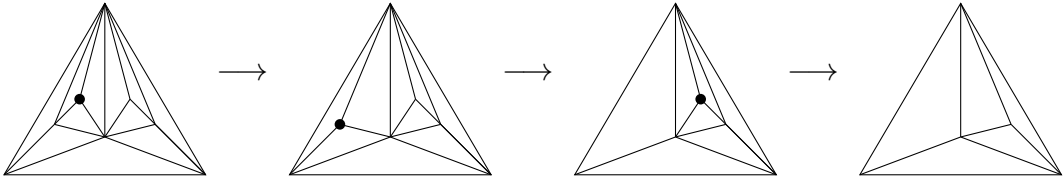
\begin{center}
\begin{tikzpicture}
\coordinate (H) at (2*2/3,1.5*0.5*2/3);
\coordinate (D) at (0,0);
\coordinate (A) at (4*0.5*2/3,6.8*0.5*2/3);
\coordinate (E) at (8*0.5*2/3,0);

\coordinate (C) at (1*2/3,1*2/3);
\coordinate (B) at (1.5*2/3,1.5*2/3);

\coordinate (G) at (2.5*2/3,1.5*2/3);
\coordinate (F) at (3*2/3,1*2/3);

\draw (H)--(A);
\draw (H)--(D);
\draw (H)--(E);

\draw (C)--(D);
\draw (C)--(A);
\draw (C)--(H);

\draw (B)--(C);
\draw (B)--(A);
\draw (B)--(H);

\draw (F)--(E);
\draw (F)--(H);
\draw (F)--(A);

\draw (G)--(A);
\draw (G)--(H);
\draw (G)--(E);

\draw (A)--(D)--(E)--cycle;

\fill[black] (B) circle (0.07);

\draw (4*2/3,4/3) node[right] {$\longrightarrow$};

\end{tikzpicture}
\begin{tikzpicture}
\coordinate (H) at (2*2/3,1.5*0.5*2/3);
\coordinate (D) at (0,0);
\coordinate (A) at (4*0.5*2/3,6.8*0.5*2/3);
\coordinate (E) at (8*0.5*2/3,0);

\coordinate (C) at (1*2/3,1*2/3);
\coordinate (B) at (1.5*2/3,1.5*2/3);

\coordinate (G) at (2.5*2/3,1.5*2/3);
\coordinate (F) at (3*2/3,1*2/3);

\draw (H)--(A);
\draw (H)--(D);
\draw (H)--(E);

\draw (C)--(D);
\draw (C)--(A);
\draw (C)--(H);

\draw (F)--(E);
\draw (F)--(H);
\draw (F)--(A);

\draw (G)--(A);
\draw (G)--(H);
\draw (G)--(E);

\draw (A)--(D)--(E)--cycle;

\fill[black] (C) circle (0.07);

\draw (4*2/3,4/3) node[right] {$\longrightarrow$};

\end{tikzpicture}
\begin{tikzpicture}
\coordinate (H) at (2*2/3,1.5*0.5*2/3);
\coordinate (D) at (0,0);
\coordinate (A) at (4*0.5*2/3,6.8*0.5*2/3);
\coordinate (E) at (8*0.5*2/3,0);

\coordinate (C) at (1*2/3,1*2/3);
\coordinate (B) at (1.5*2/3,1.5*2/3);

\coordinate (G) at (2.5*2/3,1.5*2/3);
\coordinate (F) at (3*2/3,1*2/3);

\draw (H)--(A);
\draw (H)--(D);
\draw (H)--(E);

\draw (F)--(E);
\draw (F)--(H);
\draw (F)--(A);

\draw (G)--(A);
\draw (G)--(H);
\draw (G)--(E);

\draw (A)--(D)--(E)--cycle;

\fill[black] (G) circle (0.07);

\draw (4*2/3,4/3) node[right] {$\longrightarrow$};

\end{tikzpicture}
\begin{tikzpicture}
\coordinate (H) at (2*2/3,1.5*0.5*2/3);
\coordinate (D) at (0,0);
\coordinate (A) at (4*0.5*2/3,6.8*0.5*2/3);
\coordinate (E) at (8*0.5*2/3,0);

\coordinate (C) at (1*2/3,1*2/3);
\coordinate (B) at (1.5*2/3,1.5*2/3);

\coordinate (G) at (2.5*2/3,1.5*2/3);
\coordinate (F) at (3*2/3,1*2/3);

\draw (H)--(A);
\draw (H)--(D);
\draw (H)--(E);

\draw (F)--(E);
\draw (F)--(H);
\draw (F)--(A);

\draw (A)--(D)--(E)--cycle;
\end{tikzpicture}
\end{center}
of fans in $\mathbb{R}^3$. The morphisms 
$\varphi_1$, $\varphi_2$ and $\varphi_3$ are associated to 
the star subdivisions along 
\[
v_2=\frac{1}{2}(v_1+v_3+v_8),\ v_3=\frac{1}{3}(v_1+2v_4+v_8)\mbox{ and }
v_7=\frac{1}{3}(v_1+v_6+v_8)
\] 
of fans, respectively.
Namely, 
$\varphi_1$, $\varphi_2$ and $\varphi_3$ are divisorial contractions which contract 
$V(\<v_2\>)\subset Z_2(a)$, $V(\<v_3\>)\subset Y_1$ and $V(\<v_7\>)\subset Y_2$, respectively. 
Since $Y_3$ is $\mathbb{Q}$-factorial and its Picard number is two, 
$Y_3$, $Y_2$, $Y_1$ and $Z_2(a)$ are projective (see Remark \ref{p-rem1.6}). 

\bigskip

\noindent
{\bf [8-10]}\quad 
Let $Z_{10}:=X_\Sigma$ be the smooth complete toric threefold associated to 
the fan $\Sigma$ in $\mathbb{R}^3$ whose one-dimensional cones are 
generated by 
\[
v_1 = (0,0,1),  v_2 = (-1,-1,-1), 
v_3 = (0,-1,-2),v_4 = (0,0,-1), 
\]
\[
v_5 = (0,1,0),  v_6 = (1,0,-1), 
v_7 = (1,0,0),  v_8 = (1,-1,-2),
\]
and the three-dimensional cones of $\Sigma$ are 
\[
\<v_1, v_2, v_5 \>,  \<v_1, v_2, v_7 \>,  \<v_1, v_5, v_7 \>,  \<v_2, v_3, v_4\>, 
\<v_2, v_3, v_8 \>,  \<v_2, v_4, v_5 \>, 
\]
\[
 \<v_2, v_7, v_8 \>,  \<v_3, v_4, v_8 \>, 
\<v_4, v_5, v_8 \>,  \<v_5, v_6, v_7 \>,  \<v_5, v_6, v_8 \>,  \<v_6, v_7, v_8\>.
\]
The configuration of cones of $\Sigma$ are as the following diagram: 
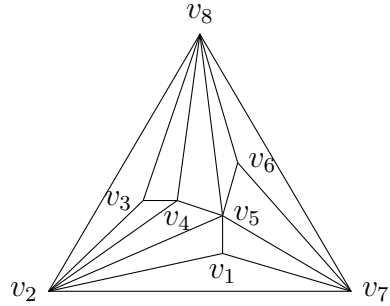
\begin{center}
\begin{tikzpicture}

\coordinate (A) at (2.3,0.5);
\coordinate (B) at (0,0);
\coordinate (H) at (4*0.5,6.8*0.5);
\coordinate (G) at (8*0.5,0);

\coordinate (C) at (1.25,1.2);
\coordinate (E) at (2.3,1.0);

\coordinate (D) at (1.7,1.2);

\coordinate (F) at (2.5,1.7);

\draw (E)--(B);
\draw (E)--(G);
\draw (E)--(H);

\draw (A)--(B);
\draw (A)--(E);
\draw (A)--(G);

\draw (F)--(E);
\draw (F)--(G);
\draw (F)--(H);

\draw (C)--(B);
\draw (C)--(D);
\draw (C)--(H);

\draw (D)--(H);
\draw (D)--(E);
\draw (D)--(B);

\draw (B)--(G)--(H)--cycle;

\draw (A) node[below] {$v_1$};
\draw (B) node[left] {$v_2$};
\draw (C) node[left] {$v_3$};
\draw (D) node[below] {$v_4$};
\draw (E) node[right] {$v_5$};
\draw (F) node[right] {$v_6$};

\draw (G) node[right] {$v_7$};
\draw (H) node[above] {$v_8$};

\end{tikzpicture}
\end{center}
There is the extra three-dimensional cone 
$\langle v_2,v_7,v_8\rangle$. 

We obtain the sequence 
\[
Z_{10}\stackrel{\varphi_1}{\longrightarrow} 
Y_1\stackrel{\varphi_2}{\longrightarrow} Y_2\stackrel{\varphi_3}{\longrightarrow} Y_3
\]
of birational morphisms associated to the diagrams 
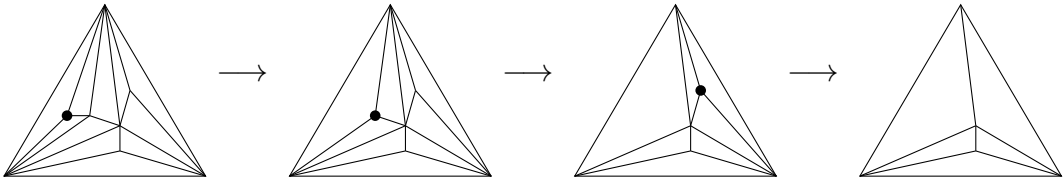
\begin{center}
\begin{tikzpicture}
\coordinate (A) at (2.3*2/3,0.5*2/3);
\coordinate (B) at (0,0);
\coordinate (H) at (4*0.5*2/3,6.8*0.5*2/3);
\coordinate (G) at (8*0.5*2/3,0);

\coordinate (C) at (1.25*2/3,1.2*2/3);
\coordinate (E) at (2.3*2/3,1.0*2/3);

\coordinate (D) at (1.7*2/3,1.2*2/3);

\coordinate (F) at (2.5*2/3,1.7*2/3);

\draw (E)--(B);
\draw (E)--(G);
\draw (E)--(H);

\draw (A)--(B);
\draw (A)--(E);
\draw (A)--(G);

\draw (F)--(E);
\draw (F)--(G);
\draw (F)--(H);

\draw (C)--(B);
\draw (C)--(D);
\draw (C)--(H);

\draw (D)--(H);
\draw (D)--(E);
\draw (D)--(B);

\draw (B)--(G)--(H)--cycle;

\fill[black] (C) circle (0.07);

\draw (4*2/3,4/3) node[right] {$\longrightarrow$};

\end{tikzpicture}
\begin{tikzpicture}
\coordinate (A) at (2.3*2/3,0.5*2/3);
\coordinate (B) at (0,0);
\coordinate (H) at (4*0.5*2/3,6.8*0.5*2/3);
\coordinate (G) at (8*0.5*2/3,0);

\coordinate (C) at (1.25*2/3,1.2*2/3);
\coordinate (E) at (2.3*2/3,1.0*2/3);

\coordinate (D) at (1.7*2/3,1.2*2/3);

\coordinate (F) at (2.5*2/3,1.7*2/3);

\draw (E)--(B);
\draw (E)--(G);
\draw (E)--(H);

\draw (A)--(B);
\draw (A)--(E);
\draw (A)--(G);

\draw (F)--(E);
\draw (F)--(G);
\draw (F)--(H);

\draw (D)--(H);
\draw (D)--(E);
\draw (D)--(B);

\draw (B)--(G)--(H)--cycle;

\fill[black] (D) circle (0.07);

\draw (4*2/3,4/3) node[right] {$\longrightarrow$};

\end{tikzpicture}
\begin{tikzpicture}
\coordinate (A) at (2.3*2/3,0.5*2/3);
\coordinate (B) at (0,0);
\coordinate (H) at (4*0.5*2/3,6.8*0.5*2/3);
\coordinate (G) at (8*0.5*2/3,0);

\coordinate (C) at (1.25*2/3,1.2*2/3);
\coordinate (E) at (2.3*2/3,1.0*2/3);

\coordinate (D) at (1.7*2/3,1.2*2/3);

\coordinate (F) at (2.5*2/3,1.7*2/3);

\draw (E)--(B);
\draw (E)--(G);
\draw (E)--(H);

\draw (A)--(B);
\draw (A)--(E);
\draw (A)--(G);

\draw (F)--(E);
\draw (F)--(G);
\draw (F)--(H);

\draw (B)--(G)--(H)--cycle;

\fill[black] (F) circle (0.07);

\draw (4*2/3,4/3) node[right] {$\longrightarrow$};

\end{tikzpicture}
\begin{tikzpicture}
\coordinate (A) at (2.3*2/3,0.5*2/3);
\coordinate (B) at (0,0);
\coordinate (H) at (4*0.5*2/3,6.8*0.5*2/3);
\coordinate (G) at (8*0.5*2/3,0);

\coordinate (C) at (1.25*2/3,1.2*2/3);
\coordinate (E) at (2.3*2/3,1.0*2/3);

\coordinate (D) at (1.7*2/3,1.2*2/3);

\coordinate (F) at (2.5*2/3,1.7*2/3);

\draw (E)--(B);
\draw (E)--(G);
\draw (E)--(H);

\draw (A)--(B);
\draw (A)--(E);
\draw (A)--(G);

\draw (B)--(G)--(H)--cycle;

\end{tikzpicture}
\end{center}
of fans in $\mathbb{R}^3$. The morphisms 
$\varphi_1$, $\varphi_2$ and $\varphi_3$ are associated to 
the star subdivisions along 
\[
v_3=\frac{1}{2}\left(v_2+v_4+v_8\right),\ v_4=\frac{1}{3}\left(v_2+2v_5+v_8\right)\mbox{ and }
v_6=\frac{1}{2}\left(v_5+v_7+v_8\right)
\] 
of fans, respectively. 
Namely, 
$\varphi_1$, $\varphi_2$ and $\varphi_3$ are divisorial contractions which contract 
$V(\<v_3\>)\subset Z_{10}$, $V(\<v_4\>)\subset Y_1$ and $V(\<v_6\>)\subset Y_2$, respectively. 
Since $Y_3$ is $\mathbb{Q}$-factorial and its Picard number is two, 
$Y_3$, $Y_2$, $Y_1$ and $Z_{10}$ are projective (see Remark \ref{p-rem1.6}). 

\bigskip

\noindent
{\bf [8-11]}\quad 
Let $a,b\in\mathbb{Z}$ and $Z_{11}(a,b):=X_\Sigma$ 
the smooth complete toric threefold associated to 
the fan $\Sigma$ in $\mathbb{R}^3$ whose one-dimensional cones are 
generated by 
\[
v_1 = (1,0,0),  v_2 = (0,-1,0), 
v_3 = (0,0,1),v_4 = (1,1,a), 
\]
\[
v_5 = (0,0,-1),  v_6 = (0,-1,-1), 
v_7 = (-1,-2,-1),  v_8 = (0,1,b),
\]
and the three-dimensional cones of $\Sigma$ are 
\[
\<v_1, v_2, v_3 \>,  \<v_1, v_3, v_4 \>,  \<v_1, v_4, v_5 \>,  \<v_1, v_5, v_6\>, 
\<v_1, v_6, v_7 \>,  \<v_1, v_2, v_7 \>, 
\]
\[
 \<v_2, v_3, v_7 \>,  \<v_5, v_6, v_7 \>, 
\<v_3, v_4, v_8 \>,  \<v_4, v_5, v_8 \>,  \<v_5, v_7, v_8 \>,  \<v_3, v_7, v_8\>.
\]
The configuration of cones of $\Sigma$ are as the following diagram: 
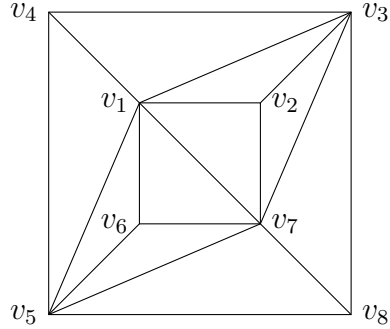
\begin{center}
\begin{tikzpicture}

\coordinate (F) at (1.2,1.2);
\coordinate (E) at (0,0);
\coordinate (H) at (4,0);
\coordinate (G) at (2.8,1.2);
\coordinate (B) at (2.8,2.8);
\coordinate (A) at (1.2,2.8);
\coordinate (D) at (0,4);
\coordinate (C) at (4,4);

\draw (A)--(G);

\draw (D)--(A);

\draw (H)--(G);

\draw (C)--(A);
\draw (C)--(B);
\draw (C)--(G);

\draw (E)--(A);
\draw (E)--(F);
\draw (E)--(G);

\draw (A)--(B)--(G)--(F)--cycle;
\draw (D)--(C)--(H)--(E)--cycle;

\draw (A) node[left] {$v_1$};
\draw (B) node[right] {$v_2$};
\draw (C) node[right] {$v_3$};
\draw (D) node[left] {$v_4$};
\draw (E) node[left] {$v_5$};
\draw (F) node[left] {$v_6$};

\draw (G) node[right] {$v_7$};
\draw (H) node[right] {$v_8$};

\end{tikzpicture}
\end{center}
In this case, there are two extra three-dimensional cones 
$\langle v_4,v_5,v_8\rangle$ and $\langle v_3,v_4,v_8\rangle$. 

We obtain the sequence 
\[
Z_{11}(a,b)\stackrel{\varphi_1}{\longrightarrow} 
Y_1\stackrel{\varphi_2}{\longrightarrow} Y_2
\]
of birational morphisms associated to the diagrams 
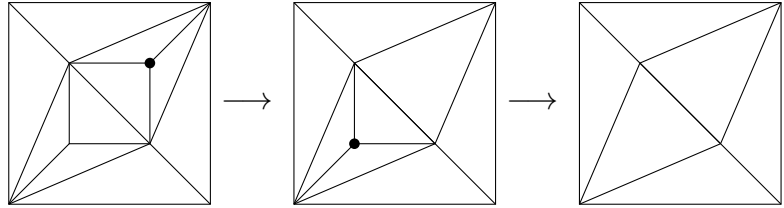
\begin{center}
\begin{tikzpicture}
\coordinate (F) at (1.2*2/3,1.2*2/3);
\coordinate (E) at (0,0);
\coordinate (H) at (4*2/3,0);
\coordinate (G) at (2.8*2/3,1.2*2/3);
\coordinate (B) at (2.8*2/3,2.8*2/3);
\coordinate (A) at (1.2*2/3,2.8*2/3);
\coordinate (D) at (0,4*2/3);
\coordinate (C) at (4*2/3,4*2/3);

\draw (A)--(G);

\draw (D)--(A);

\draw (H)--(G);

\draw (C)--(A);
\draw (C)--(B);
\draw (C)--(G);

\draw (E)--(A);
\draw (E)--(F);
\draw (E)--(G);

\draw (A)--(B)--(G)--(F)--cycle;
\draw (D)--(C)--(H)--(E)--cycle;

\fill[black] (B) circle (0.07);

\draw (4*2/3,4/3) node[right] {$\longrightarrow$};

\end{tikzpicture}
\begin{tikzpicture}
\coordinate (F) at (1.2*2/3,1.2*2/3);
\coordinate (E) at (0,0);
\coordinate (H) at (4*2/3,0);
\coordinate (G) at (2.8*2/3,1.2*2/3);
\coordinate (B) at (2.8*2/3,2.8*2/3);
\coordinate (A) at (1.2*2/3,2.8*2/3);
\coordinate (D) at (0,4*2/3);
\coordinate (C) at (4*2/3,4*2/3);

\draw (A)--(G);

\draw (D)--(A);

\draw (H)--(G);

\draw (C)--(A);

\draw (C)--(G);

\draw (E)--(A);
\draw (E)--(F);
\draw (E)--(G);

\draw (A)--(G)--(F)--cycle;
\draw (D)--(C)--(H)--(E)--cycle;

\fill[black] (F) circle (0.07);

\draw (4*2/3,4/3) node[right] {$\longrightarrow$};

\end{tikzpicture}
\begin{tikzpicture}
\coordinate (F) at (1.2*2/3,1.2*2/3);
\coordinate (E) at (0,0);
\coordinate (H) at (4*2/3,0);
\coordinate (G) at (2.8*2/3,1.2*2/3);
\coordinate (B) at (2.8*2/3,2.8*2/3);
\coordinate (A) at (1.2*2/3,2.8*2/3);
\coordinate (D) at (0,4*2/3);
\coordinate (C) at (4*2/3,4*2/3);

\draw (A)--(G);

\draw (D)--(A);

\draw (H)--(G);

\draw (C)--(A);

\draw (C)--(G);

\draw (E)--(A);

\draw (E)--(G);

\draw (A)--(G);
\draw (D)--(C)--(H)--(E)--cycle;

\end{tikzpicture}
\end{center}
of fans in $\mathbb{R}^3$. The morphisms 
$\varphi_1$ and $\varphi_2$ are associated to 
the star subdivisions along 
\[
v_2=\frac{1}{2}(v_1+v_3+v_7)\mbox{ and }
v_6=\frac{1}{2}(v_1+v_5+v_7)
\] 
of fans, respectively. 
Namely, 
$\varphi_1$ and $\varphi_2$ are divisorial contractions which contract 
$V(\<v_2\>)\subset Z_{11}(a,b)$ and $V(\<v_6\>)\subset Y_1$, respectively. 
Since $v_3+v_5=0$, there exists a surjective toric morphism $\phi:Y_2\to S$ 
whose every 
fiber is isomorphic to $\mathbb{P}^1$ set-theoretically, where $S$ is 
a complete toric surface. 
Therefore, $\phi$ is a projective morphism since we can easily find 
a $\phi$-ample Cartier divisor on $Y$. 
Thus, $Y_2$, $Y_1$ and $Z_{11}(a,b)$ are projective. 

\bigskip

From now, 
we confirm our Conjecture \ref{p-conj1.4} for remaining eight 
types of 
smooth complete toric threefolds of Picard number five. 
The varieties of these types are non-projective except for a few values. 
In fact, for a few values of parameters $a,b,c,d\in\mathbb{Z}$ below, the varieties 
become projective. 

\bigskip

\noindent
{\bf [8-5$'$]}\quad 
Let $a\in\mathbb{Z}$ and $Z'_5(a):=X_\Sigma$ the smooth complete toric threefold associated to 
the fan $\Sigma$ in $\mathbb{R}^3$ whose one-dimensional cones are 
generated by 
\[
v_1=(1,0,0),v_2=(0,1,0),v_3=(0,0,1),v_4=(0,-1,-a),
\]
\[
v_5=(0,0,-1),v_6=(-1,1,-1),v_7=(-1,0,-1),v_8=(-1,-1,0),
\]
and the three-dimensional cones of $\Sigma$ are 
\[
   \< v_{1}, v_{2}, v_{3} \>,  \< v_{1}, v_{3}, v_{4} \>, 
    \< v_{1}, v_{4}, v_{5} \>,  \< v_{1}, v_{5}, v_{6} \>, 
    \< v_{1}, v_{2}, v_{6} \>,  \< v_{2}, v_{3}, v_{8} \>,
    \]
    \[
    \< v_{3}, v_{4}, v_{8} \>,  \< v_{4}, v_{5}, v_{8} \>, 
    \<v_5, v_6, v_7\>,  \<v_5, v_7, v_8 \>, 
    \<v_6, v_7, v_8\>,  \<v_2, v_6, v_8 \>.
\]
The configuration of cones of $\Sigma$ are as the following diagram: 
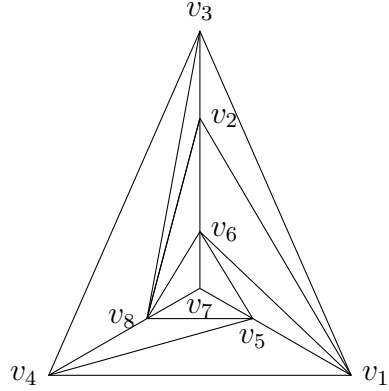
\begin{center}
\begin{tikzpicture}

\coordinate (F) at (4*0.5,3.8*0.5);
\coordinate (H) at (2.6*0.5,1.5*0.5);
\coordinate (E) at (5.4*0.5,1.5*0.5);
\coordinate (C) at (2,6.8*0.5+1.15);
\coordinate (D) at (0,0);
\coordinate (B) at (4*0.5,6.8*0.5);
\coordinate (A) at (8*0.5,0);
\coordinate (G) at (2,2.3*0.5);

\draw (G)--(E);
\draw (G)--(F);
\draw (G)--(H);

\draw (A)--(B);
\draw (B)--(H);

\draw (B)--(C);
\draw (C)--(H);

\draw (A)--(E);
\draw (D)--(H);
\draw (B)--(F);

\draw (A)--(F);
\draw (B)--(H);
\draw (D)--(E);

\draw (E)--(F)--(H)--cycle;
\draw (A)--(C)--(D)--cycle;

\draw (A) node[right] {$v_1$};
\draw (B) node[right] {$v_2$};
\draw (C) node[above] {$v_3$};
\draw (D) node[left] {$v_4$};
\draw (E) node[below] {$v_5$};
\draw (F) node[right] {$v_6$};

\draw (G) node[below] {$v_7$};
\draw (H) node[left] {$v_8$};

\end{tikzpicture}
\end{center}
There is the extra three-dimensional cone 
$\langle v_1,v_3,v_4\rangle$. 

We have already proved that $Z'_5(a)$ is non-projective for 
$a\ne 0, -1$ in \cite[Example 2]{fujino-payne} (see Remark \ref{p-rem4.2} below). 
If $a=-1$, then we have primitive relations 
\[v_1+v_8=v_4+v_5, \quad 
v_4+v_6=v_2+v_8, \quad
\text{and}\quad
v_2+v_5=v_1+v_6.   
\] 
Thus $Z'_5(-1)$ is non-projective 
by the above primitive relations and Lemma \ref{p-lem2.9}. 
If $a=0$, then we have 
a primitive relation 
\[ v_1+v_8=v_4. 
\] 
Then we have a contraction $Z'_5(0)\to V$ by removing $v_4$. 
By construction, $V$ is a smooth 
complete toric threefold with Picard number four. 
Since $v_3+v_5=0$, $V$ is not isomorphic 
to $W$ in Example \ref{p-ex4.1}. 
This means that $V$ and $Z'_5(0)$ are projective. 

We obtain the sequence 
\[
Z'_5(a)\stackrel{\psi}{\dasharrow} X'\stackrel{\varphi_1}{\longrightarrow} Y_1\stackrel{\varphi_2}{\longrightarrow} Y_2
\]
of birational maps associated to the diagrams 
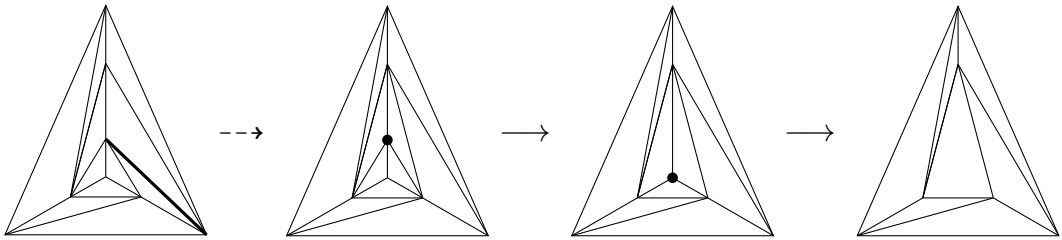
\begin{center}
\begin{tikzpicture}
\coordinate (F) at (4*0.5*2/3,3.8*0.5*2/3);
\coordinate (H) at (2.6*0.5*2/3,1.5*0.5*2/3);
\coordinate (E) at (5.4*0.5*2/3,1.5*0.5*2/3);
\coordinate (C) at (2*2/3,6.8*0.5*2/3+1.15*2/3);
\coordinate (D) at (0,0);
\coordinate (B) at (4*0.5*2/3,6.8*0.5*2/3);
\coordinate (A) at (8*0.5*2/3,0);
\coordinate (G) at (2*2/3,2.3*0.5*2/3);

\draw (G)--(E);
\draw (G)--(F);
\draw (G)--(H);

\draw (A)--(B);
\draw (B)--(H);

\draw (B)--(C);
\draw (C)--(H);

\draw (A)--(E);
\draw (D)--(H);
\draw (B)--(F);

\draw [very thick] (A)--(F);
\draw (B)--(H);
\draw (D)--(E);

\draw (E)--(F)--(H)--cycle;
\draw (A)--(C)--(D)--cycle;

\draw (8*0.5*2/3,4/3) node[right] {$\dashrightarrow$};

\end{tikzpicture}
\begin{tikzpicture}
\coordinate (F) at (4*0.5*2/3,3.8*0.5*2/3);
\coordinate (H) at (2.6*0.5*2/3,1.5*0.5*2/3);
\coordinate (E) at (5.4*0.5*2/3,1.5*0.5*2/3);
\coordinate (C) at (2*2/3,6.8*0.5*2/3+1.15*2/3);
\coordinate (D) at (0,0);
\coordinate (B) at (4*0.5*2/3,6.8*0.5*2/3);
\coordinate (A) at (8*0.5*2/3,0);
\coordinate (G) at (2*2/3,2.3*0.5*2/3);

\draw (G)--(E);
\draw (G)--(F);
\draw (G)--(H);

\draw (A)--(B);
\draw (B)--(H);

\draw (B)--(C);
\draw (C)--(H);

\draw (A)--(E);
\draw (D)--(H);
\draw (B)--(F);

\draw (B)--(E);
\draw (B)--(H);
\draw (D)--(E);

\draw (E)--(F)--(H)--cycle;
\draw (A)--(C)--(D)--cycle;

\draw (8*0.5*2/3,4/3) node[right] {$\longrightarrow$};

\fill[black] (F) circle (0.07);

\end{tikzpicture}
\begin{tikzpicture}
\coordinate (F) at (4*0.5*2/3,3.8*0.5*2/3);
\coordinate (H) at (2.6*0.5*2/3,1.5*0.5*2/3);
\coordinate (E) at (5.4*0.5*2/3,1.5*0.5*2/3);
\coordinate (C) at (2*2/3,6.8*0.5*2/3+1.15*2/3);
\coordinate (D) at (0,0);
\coordinate (B) at (4*0.5*2/3,6.8*0.5*2/3);
\coordinate (A) at (8*0.5*2/3,0);
\coordinate (G) at (2*2/3,2.3*0.5*2/3);

\draw (G)--(E);

\draw (G)--(H);

\draw (A)--(B);
\draw (B)--(H);

\draw (B)--(C);
\draw (C)--(H);

\draw (B)--(G);

\draw (A)--(E);
\draw (D)--(H);

\draw (B)--(E);
\draw (B)--(H);
\draw (D)--(E);

\draw (E)--(H);
\draw (A)--(C)--(D)--cycle;

\draw (8*0.5*2/3,4/3) node[right] {$\longrightarrow$};

\fill[black] (G) circle (0.07);

\end{tikzpicture}
\begin{tikzpicture}
\coordinate (F) at (4*0.5*2/3,3.8*0.5*2/3);
\coordinate (H) at (2.6*0.5*2/3,1.5*0.5*2/3);
\coordinate (E) at (5.4*0.5*2/3,1.5*0.5*2/3);
\coordinate (C) at (2*2/3,6.8*0.5*2/3+1.15*2/3);
\coordinate (D) at (0,0);
\coordinate (B) at (4*0.5*2/3,6.8*0.5*2/3);
\coordinate (A) at (8*0.5*2/3,0);
\coordinate (G) at (2*2/3,2.3*0.5*2/3);

\draw (A)--(B);
\draw (B)--(H);

\draw (B)--(C);
\draw (C)--(H);

\draw (A)--(E);
\draw (D)--(H);

\draw (B)--(E);
\draw (B)--(H);
\draw (D)--(E);

\draw (E)--(H);
\draw (A)--(C)--(D)--cycle;

\end{tikzpicture}
\end{center}
of fans in $\mathbb{R}^3$. The rational map 
$\psi$ is the flop associated to the primitive relation 
\[
v_2+v_5=v_1+v_6,
\]
that is, the fan corresponding to $X'$ is obtained from $\Sigma$ 
by removing $\<v_1,v_6\>$ and by adding $\<v_2,v_5\>$. 
The morphisms 
$\varphi_1$ and $\varphi_2$ are associated to 
the star subdivisions along 
\[
v_6=v_2+v_7\mbox{ and }
v_7=v_2+v_5+v_8
\] 
of fans, respectively. 
Namely, 
$\psi$  is the flop along $V(\langle v_1,v_6\rangle)$, 
$\varphi_1$ is the blow-up of $Y_1$ along the curve $V(\langle v_2,v_7\rangle)$ and 
$\varphi_2$ is the blow-up of $Y_2$ at the point $V(\langle v_2,v_5,v_8\rangle)$. 
Since $Y_2$ is smooth and its Picard number is three, 
$Y_2$, $Y_1$ and $X'$ are projective (see Remark \ref{p-rem1.6}). 
Conjecture \ref{p-conj1.4} is true for this case. 

\begin{rem}\label{p-rem4.2}
The variety $Z_5'(a)$ is studied in \cite[Example 2]{fujino-payne}, 
where we prove that $\NE(Z'_5(a))=\mathbb R^5$ holds when $a\ne 0, -1$. 
In particular, $Z'_5(a)$ is non-projective when $a\ne 0, -1$. 
We put $n=(1, 0, 0)$, $n'=(0, 1, 0)$, and 
$n''=(0, 0, 1)$. Then 
the above description of $Z'_5(a)$ coincides with 
[8-5$'$] in \cite[p.78]{oda1}. 
\end{rem}

Before addressing the other cases, we make a useful remark. 

\begin{rem}\label{p-rem4.3}
Almost all cases considered in the proof of Theorem \ref{p-thm1.5} follow 
a common strategy:~after taking a flop, we demonstrate 
the projectivity of the resulting variety by constructing a sequence 
of projective contraction morphisms terminating at a variety that is necessarily projective. 

We have already seen this method applied to the toric variety labeled 
as [8-5$'$] in \cite{oda1}. We will show 
that the same strategy is applicable 
to [8-5$'$], [8-5$''$], [8-8], [8-12], [8-14$'$], and [8-14$''$]. 

However, we note that additional work is required 
in the cases of [8-13$'$] and [8-13$''$]. 
\end{rem}

\bigskip

\noindent
{\bf [8-5$''$]}\quad 
Let $Z''_5:=X_\Sigma$ be the smooth complete toric threefold associated to 
the fan $\Sigma$ in $\mathbb{R}^3$ whose one-dimensional cones are 
generated by 
\[
v_1 = (0,1,0), v_2 = (0,-1,-1), 
 v_3 = (1,0,0),  v_4 = (0,0,1), 
\]
\[
  v_5 = (-1,0,-1),  v_6 = (-1,-2,-2), 
v_7 = (-1,-1,-1),  v_8 = (-1,-1,0), 
\]
and the three-dimensional cones of $\Sigma$ are 
\[
   \< v_{1}, v_{2}, v_{3} \>,  \< v_{1}, v_{3}, v_{4} \>, 
    \< v_{1}, v_{4}, v_{5} \>,  \< v_{1}, v_{5}, v_{6} \>, 
    \< v_{1}, v_{2}, v_{6} \>,  \< v_{2}, v_{3}, v_{8} \>, 
 \]
 \[
    \< v_{3}, v_{4}, v_{8} \>,  \< v_{4}, v_{5}, v_{8} \>, 
    \<v_5, v_6, v_7\>,  \<v_5, v_7, v_8 \>, 
    \<v_6, v_7, v_8\>,  \<v_2, v_6, v_8 \>.
\]
The configuration of cones of $\Sigma$ are as the following diagram: 
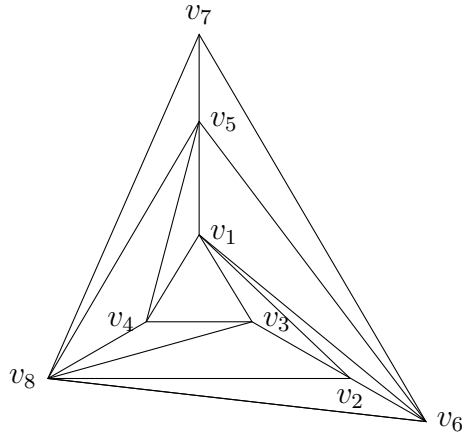
\begin{center}
\begin{tikzpicture}

\coordinate (A) at (4*0.5,3.8*0.5);
\coordinate (D) at (2.6*0.5,1.5*0.5);
\coordinate (C) at (5.4*0.5,1.5*0.5);
\coordinate (G) at (2,6.8*0.5+1.15);
\coordinate (H) at (0,0);
\coordinate (E) at (4*0.5,6.8*0.5);
\coordinate (B) at (8*0.5,0);
\coordinate (F) at (4+2*0.5,-1.14*0.5);

\draw (A)--(G);
\draw (D)--(H);
\draw (C)--(F);

\draw (D)--(E);
\draw (A)--(F);
\draw (A)--(B);

\draw (B)--(H);
\draw (C)--(H);

\draw (A)--(C)--(D)--cycle;
\draw (E)--(F)--(H)--cycle;
\draw (F)--(G)--(H)--cycle;

\draw (A) node[right] {$v_1$};
\draw (B) node[below] {$v_2$};
\draw (C) node[right] {$v_3$};
\draw (D) node[left] {$v_4$};
\draw (E) node[right] {$v_5$};
\draw (F) node[right] {$v_6$};

\draw (G) node[above] {$v_7$};
\draw (H) node[left] {$v_8$};

\end{tikzpicture}
\end{center}
There is the extra three-dimensional cone 
$\langle v_6,v_7,v_8\rangle$. 

We have primitive relations 
\[v_1+v_8=v_4+v_5, \quad 
v_2+v_4=v_3+v_8, \quad 
\text{and} 
\quad 
v_3+v_5=v_1+v_2. 
\] 
Thus we obtain that $Z''_5$ is non-projective 
by the above primitive relations and Lemma \ref{p-lem2.9}. 

We obtain the sequence 
\[
Z''_5\stackrel{\psi}{\dasharrow} X'\stackrel{\varphi_1}{\longrightarrow} Y_1\stackrel{\varphi_2}{\longrightarrow} Y_2
\]
of birational maps associated to the diagrams 
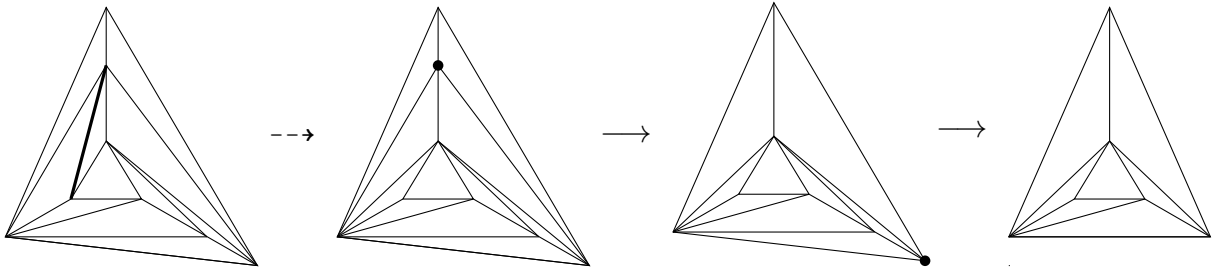
\begin{center}
\begin{tikzpicture}
\coordinate (A) at (4*0.5*2/3,3.8*0.5*2/3);
\coordinate (D) at (2.6*0.5*2/3,1.5*0.5*2/3);
\coordinate (C) at (5.4*0.5*2/3,1.5*0.5*2/3);
\coordinate (G) at (2*2/3,6.8*0.5*2/3+1.15*2/3);
\coordinate (H) at (0,0);
\coordinate (E) at (4*0.5*2/3,6.8*0.5*2/3);
\coordinate (B) at (8*0.5*2/3,0);
\coordinate (F) at (4*2/3+2*0.5*2/3,-1.14*0.5*2/3);

\draw (A)--(G);
\draw (D)--(H);
\draw (C)--(F);

\draw [very thick] (D)--(E);
\draw (A)--(F);
\draw (A)--(B);

\draw (B)--(H);
\draw (C)--(H);

\draw (A)--(C)--(D)--cycle;
\draw (E)--(F)--(H)--cycle;
\draw (F)--(G)--(H)--cycle;

\draw (10/3,4/3) node[right] {$\dashrightarrow$};

\end{tikzpicture}
\begin{tikzpicture}
\coordinate (A) at (4*0.5*2/3,3.8*0.5*2/3);
\coordinate (D) at (2.6*0.5*2/3,1.5*0.5*2/3);
\coordinate (C) at (5.4*0.5*2/3,1.5*0.5*2/3);
\coordinate (G) at (2*2/3,6.8*0.5*2/3+1.15*2/3);
\coordinate (H) at (0,0);
\coordinate (E) at (4*0.5*2/3,6.8*0.5*2/3);
\coordinate (B) at (8*0.5*2/3,0);
\coordinate (F) at (4*2/3+2*0.5*2/3,-1.14*0.5*2/3);

\draw (A)--(G);
\draw (D)--(H);
\draw (C)--(F);

\draw (A)--(H);
\draw (A)--(F);
\draw (A)--(B);

\draw (B)--(H);
\draw (C)--(H);

\draw (A)--(C)--(D)--cycle;
\draw (E)--(F)--(H)--cycle;
\draw (F)--(G)--(H)--cycle;

\fill[black] (E) circle (0.07);

\draw (10/3,4/3) node[right] {$\longrightarrow$};

\end{tikzpicture}
\begin{tikzpicture}
\coordinate (A) at (4*0.5*2/3,3.8*0.5*2/3);
\coordinate (D) at (2.6*0.5*2/3,1.5*0.5*2/3);
\coordinate (C) at (5.4*0.5*2/3,1.5*0.5*2/3);
\coordinate (G) at (2*2/3,6.8*0.5*2/3+1.15*2/3);
\coordinate (H) at (0,0);
\coordinate (E) at (4*0.5*2/3,6.8*0.5*2/3);
\coordinate (B) at (8*0.5*2/3,0);
\coordinate (F) at (4*2/3+2*0.5*2/3,-1.14*0.5*2/3);

\draw (A)--(G);
\draw (D)--(H);
\draw (C)--(F);

\draw (A)--(H);
\draw (A)--(F);
\draw (A)--(B);

\draw (B)--(H);
\draw (C)--(H);

\draw (A)--(C)--(D)--cycle;

\draw (F)--(G)--(H)--cycle;

\fill[black] (F) circle (0.07);

\draw (10/3,4/3) node[right] {$\longrightarrow$};

\end{tikzpicture}
\begin{tikzpicture}
\coordinate (A) at (4*0.5*2/3,3.8*0.5*2/3);
\coordinate (D) at (2.6*0.5*2/3,1.5*0.5*2/3);
\coordinate (C) at (5.4*0.5*2/3,1.5*0.5*2/3);
\coordinate (G) at (2*2/3,6.8*0.5*2/3+1.15*2/3);
\coordinate (H) at (0,0);
\coordinate (E) at (4*0.5*2/3,6.8*0.5*2/3);
\coordinate (B) at (8*0.5*2/3,0);
\coordinate (F) at (4*2/3+2*0.5*2/3,-1.14*0.5*2/3);

\draw (A)--(G);
\draw (D)--(H);
\draw (C)--(B);

\draw (A)--(H);

\draw (A)--(B);

\draw (B)--(H);
\draw (C)--(H);

\draw (A)--(C)--(D)--cycle;

\draw (B)--(G)--(H)--cycle;

\draw [opacity=0] (0,-1.14/3)--(0,-1.14/3);

\end{tikzpicture}
\end{center}
of fans in $\mathbb{R}^3$. The rational map 
$\psi$ is the flop associated to the primitive relation 
\[
v_1+v_8=v_4+v_5,
\]
that is, the fan corresponding to $X'$ is obtained from $\Sigma$ 
by removing $\<v_4,v_5\>$ and by adding $\<v_1,v_8\>$. The morphisms 
$\varphi_1$ and $\varphi_2$ are associated to 
the star subdivisions along 
\[
v_5=v_1+v_7\mbox{ and }
v_6=v_2+v_7
\] 
of fans, respectively. 
Namely, $\psi$  is the flop along $V(\langle v_4,v_5\rangle)$, 
$\varphi_1$ is the blow-up of $Y_1$ along the curve $V(\langle v_1,v_7\rangle)$ and 
$\varphi_2$ is the blow-up of $Y_2$ along the curve $V(\langle v_2,v_7\rangle)$. 
Since $Y_2$ is smooth and its Picard number is three, 
$Y_2$, $Y_1$ and $X'$ are projective (see Remark \ref{p-rem1.6}). 
Conjecture \ref{p-conj1.4} is true for this case. 

\begin{rem}\label{p-rem4.4}
Let $Z''_5\dashrightarrow Z'$ be the 
flop along $V(\langle v_1,v_6\rangle)$ associated to the diagrams: 
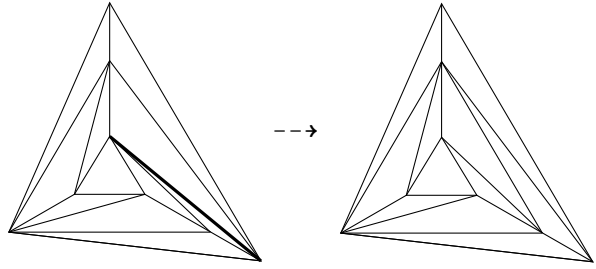
\begin{center}
\begin{tikzpicture}
\coordinate (A) at (4*0.5*2/3,3.8*0.5*2/3);
\coordinate (D) at (2.6*0.5*2/3,1.5*0.5*2/3);
\coordinate (C) at (5.4*0.5*2/3,1.5*0.5*2/3);
\coordinate (G) at (2*2/3,6.8*0.5*2/3+1.15*2/3);
\coordinate (H) at (0,0);
\coordinate (E) at (4*0.5*2/3,6.8*0.5*2/3);
\coordinate (B) at (8*0.5*2/3,0);
\coordinate (F) at (4*2/3+2*0.5*2/3,-1.14*0.5*2/3);

\draw (A)--(G);
\draw (D)--(H);
\draw (C)--(F);

\draw (D)--(E);
\draw [very thick] (A)--(F);
\draw (A)--(B);

\draw (B)--(H);
\draw (C)--(H);

\draw (A)--(C)--(D)--cycle;
\draw (E)--(F)--(H)--cycle;
\draw (F)--(G)--(H)--cycle;

\draw (10/3,4/3) node[right] {$\dashrightarrow$};

\end{tikzpicture}
\begin{tikzpicture}
\coordinate (A) at (4*0.5*2/3,3.8*0.5*2/3);
\coordinate (D) at (2.6*0.5*2/3,1.5*0.5*2/3);
\coordinate (C) at (5.4*0.5*2/3,1.5*0.5*2/3);
\coordinate (G) at (2*2/3,6.8*0.5*2/3+1.15*2/3);
\coordinate (H) at (0,0);
\coordinate (E) at (4*0.5*2/3,6.8*0.5*2/3);
\coordinate (B) at (8*0.5*2/3,0);
\coordinate (F) at (4*2/3+2*0.5*2/3,-1.14*0.5*2/3);

\draw (A)--(G);
\draw (D)--(H);
\draw (C)--(F);

\draw (D)--(E);
\draw (B)--(E);
\draw (A)--(B);

\draw (B)--(H);
\draw (C)--(H);

\draw (A)--(C)--(D)--cycle;
\draw (E)--(F)--(H)--cycle;
\draw (F)--(G)--(H)--cycle;

\end{tikzpicture}
\end{center}
Namely, this flop is associated to the primitive relation 
\[
v_2+v_5=v_1+v_6.
\]
Let $Z'\to W'$ be the divisorial contraction 
which contracts $V(\langle v_6\rangle)$, that 
is, 
$Z'$ is obtained from $W'$ by the star subdivision along 
\[
v_6=v_2+v_7. 
\] 
Then $W'$ is isomorphic to $W$ in Example \ref{p-ex4.1}. 
This can be confirmed 
by the automorphism
\[
(x,y,z)\mapsto (z,y,x)
\]
of $\mathbb{R}^3$, and 
by the changing 
\[
1\mapsto2,2\mapsto 5,4\mapsto 1,5\mapsto 7,7\mapsto 4,8\mapsto 6
\]
of indices of $v_i$'s. 
In particular, $Z'$ is non-projective (see Remark \ref{p-rem4.8}). 
\end{rem}


\noindent
{\bf [8-8]}\quad 
Let $Z_8:=X_\Sigma$ be the smooth complete toric threefold associated to 
the fan $\Sigma$ in $\mathbb{R}^3$ whose one-dimensional cones are 
generated by 
\[
v_1 = (0,0,1), v_2 = (1,0,0), 
v_3 = (0,-1,-1), v_4 = (-1,-2,-1), 
\]
\[
v_5 = (0,1,0),  v_6 = (0,0,-1), 
v_7 = (-1,-2,-2),  v_8 = (-1,-1,-2), 
\]
and the three-dimensional cones of $\Sigma$ are 
\[
\<v_1, v_2, v_3\>,  \<v_1, v_3, v_4\>,  \<v_1, v_4, v_5\>,  \<v_1, v_2, v_5\>, 
\<v_2, v_5, v_6\>,  \<v_3, v_4, v_7\>, 
\]
\[
 \<v_2, v_3, v_8\>,  \<v_3, v_7, v_8\>, 
\<v_4, v_7, v_8\>,  \<v_4, v_5, v_8\>,  \<v_5, v_6, v_8\>,  \<v_2, v_6, v_8\>.
\]
The configuration of cones of $\Sigma$ are as the following diagram: 
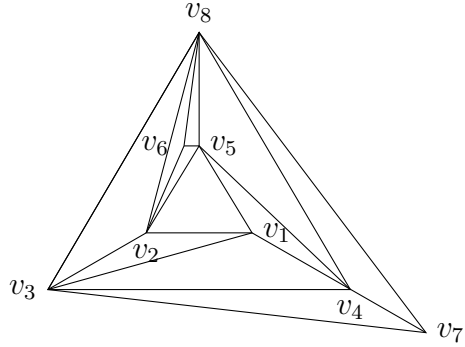
\begin{center}
\begin{tikzpicture}

\coordinate (E) at (4*0.5,3.8*0.5);
\coordinate (B) at (2.6*0.5,1.5*0.5);
\coordinate (A) at (5.4*0.5,1.5*0.5);
\coordinate (F) at (3.6*0.5,3.8*0.5);
\coordinate (C) at (0,0);
\coordinate (H) at (4*0.5,6.8*0.5);
\coordinate (D) at (8*0.5,0);
\coordinate (G) at (4+2*0.5,-1.14*0.5);

\draw (G)--(A);
\draw (B)--(C);
\draw (E)--(H);

\draw (A)--(C);
\draw (B)--(H);

\draw (D)--(E);

\draw (B)--(F);
\draw (E)--(F);
\draw (H)--(F);

\draw (A)--(B)--(E)--cycle;
\draw (C)--(D)--(H)--cycle;
\draw (C)--(G)--(H)--cycle;

\draw (A) node[right] {$v_1$};
\draw (B) node[below] {$v_2$};
\draw (C) node[left] {$v_3$};
\draw (D) node[below] {$v_4$};
\draw (E) node[right] {$v_5$};
\draw (3.5*0.5,3.8*0.5) node[left] {$v_6$};

\draw (G) node[right] {$v_7$};
\draw (H) node[above] {$v_8$};

\end{tikzpicture}
\end{center}
There is the extra three-dimensional cone 
$\langle v_3,v_7,v_8\rangle$. 

We have primitive relations 
\[
v_1+v_8=v_4+v_5, \quad 
v_2+v_4=v_1+2v_3, \quad 
v_3+v_5=v_6, \quad  
\text{and} 
\quad 
v_3+v_6=v_2+v_8. 
\] 
Hence $Z_8$ is not projective by the above primitive relations 
and Lemma \ref{p-lem2.9}. 
Here, we should remark that the relation 
\[
2v_3+v_5=v_2+v_8
\]
obtained simply from the third and fourth primitive relations above 
makes it easier to understand the non-projectiveness for $Z_8$, though 
it is not a primitive relation. 

We obtain the sequence 
\[
Z_8\stackrel{\psi}{\dasharrow} X'\stackrel{\varphi_1}{\longrightarrow} Y_1\stackrel{\varphi_2}{\longrightarrow} Y_2
\]
of birational maps associated to the diagrams 
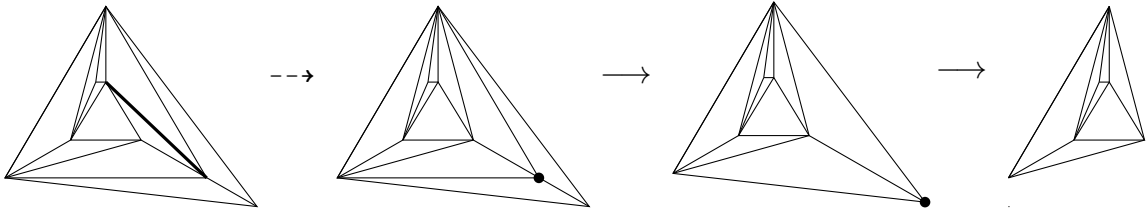
\begin{center}
\begin{tikzpicture}
\coordinate (E) at (4*0.5*2/3,3.8*0.5*2/3);
\coordinate (B) at (2.6*0.5*2/3,1.5*0.5*2/3);
\coordinate (A) at (5.4*0.5*2/3,1.5*0.5*2/3);
\coordinate (F) at (3.6*0.5*2/3,3.8*0.5*2/3);
\coordinate (C) at (0,0);
\coordinate (H) at (4*0.5*2/3,6.8*0.5*2/3);
\coordinate (D) at (8*0.5*2/3,0);
\coordinate (G) at (4*2/3+2*0.5*2/3,-1.14*0.5*2/3);

\draw (G)--(A);
\draw (B)--(C);
\draw (E)--(H);

\draw (A)--(C);
\draw (B)--(H);

\draw [very thick] (D)--(E);

\draw (B)--(F);
\draw (E)--(F);
\draw (H)--(F);

\draw (A)--(B)--(E)--cycle;
\draw (C)--(D)--(H)--cycle;
\draw (C)--(G)--(H)--cycle;

\draw (10/3,4/3) node[right] {$\dashrightarrow$};

\end{tikzpicture}
\begin{tikzpicture}
\coordinate (E) at (4*0.5*2/3,3.8*0.5*2/3);
\coordinate (B) at (2.6*0.5*2/3,1.5*0.5*2/3);
\coordinate (A) at (5.4*0.5*2/3,1.5*0.5*2/3);
\coordinate (F) at (3.6*0.5*2/3,3.8*0.5*2/3);
\coordinate (C) at (0,0);
\coordinate (H) at (4*0.5*2/3,6.8*0.5*2/3);
\coordinate (D) at (8*0.5*2/3,0);
\coordinate (G) at (4*2/3+2*0.5*2/3,-1.14*0.5*2/3);

\draw (G)--(A);
\draw (B)--(C);
\draw (E)--(H);

\draw (A)--(C);
\draw (B)--(H);

\draw (A)--(H);

\draw (B)--(F);
\draw (E)--(F);
\draw (H)--(F);

\draw (A)--(B)--(E)--cycle;
\draw (C)--(D)--(H)--cycle;
\draw (C)--(G)--(H)--cycle;

\fill[black] (D) circle (0.07);

\draw (10/3,4/3) node[right] {$\longrightarrow$};

\end{tikzpicture}
\begin{tikzpicture}
\coordinate (E) at (4*0.5*2/3,3.8*0.5*2/3);
\coordinate (B) at (2.6*0.5*2/3,1.5*0.5*2/3);
\coordinate (A) at (5.4*0.5*2/3,1.5*0.5*2/3);
\coordinate (F) at (3.6*0.5*2/3,3.8*0.5*2/3);
\coordinate (C) at (0,0);
\coordinate (H) at (4*0.5*2/3,6.8*0.5*2/3);
\coordinate (D) at (8*0.5*2/3,0);
\coordinate (G) at (4*2/3+2*0.5*2/3,-1.14*0.5*2/3);

\draw (G)--(A);
\draw (B)--(C);
\draw (E)--(H);

\draw (A)--(C);
\draw (B)--(H);

\draw (A)--(H);

\draw (B)--(F);
\draw (E)--(F);
\draw (H)--(F);

\draw (A)--(B)--(E)--cycle;
\draw (C)--(H);
\draw (C)--(G)--(H)--cycle;

\fill[black] (G) circle (0.07);

\draw (10/3,4/3) node[right] {$\longrightarrow$};

\end{tikzpicture}
\begin{tikzpicture}
\coordinate (E) at (4*0.5*2/3,3.8*0.5*2/3);
\coordinate (B) at (2.6*0.5*2/3,1.5*0.5*2/3);
\coordinate (A) at (5.4*0.5*2/3,1.5*0.5*2/3);
\coordinate (F) at (3.6*0.5*2/3,3.8*0.5*2/3);
\coordinate (C) at (0,0);
\coordinate (H) at (4*0.5*2/3,6.8*0.5*2/3);
\coordinate (D) at (8*0.5*2/3,0);
\coordinate (G) at (4*2/3+2*0.5*2/3,-1.14*0.5*2/3);

\draw (B)--(C);
\draw (E)--(H);

\draw (A)--(C);
\draw (B)--(H);

\draw (A)--(H);

\draw (B)--(F);
\draw (E)--(F);
\draw (H)--(F);

\draw (A)--(B)--(E)--cycle;
\draw (C)--(H);
\draw (C)--(H);

\draw [opacity=0] (0,-1.14/3)--(0,-1.14/3);

\end{tikzpicture}
\end{center}
of fans in $\mathbb{R}^3$. The rational map 
$\psi$ is the flop associated to the primitive relation 
\[
v_1+v_8=v_4+v_5,
\]
that is, the fan corresponding to $X'$ is obtained from $\Sigma$ 
by removing $\<v_4,v_5\>$ and by adding $\<v_1,v_8\>$. 
The morphisms 
$\varphi_1$ and $\varphi_2$ are associated to 
the star subdivisions along 
\[
v_4=v_1+v_7\mbox{ and }
v_7=v_1+v_3+v_8
\] 
of fans, respectively. 
Namely, $\psi$  is the flop along $V(\langle v_4,v_5\rangle)$, 
$\varphi_1$ is the blow-up of $Y_1$ along the curve $V(\langle v_1,v_7\rangle)$ and 
$\varphi_2$ is the blow-up of $Y_2$ at the point $V(\langle v_1,v_3,v_8\rangle)$. 
Since $Y_2$ is smooth and its Picard number is three, 
$Y_2$, $Y_1$ and $X'$ are projective (see Remark \ref{p-rem1.6}). 
Conjecture \ref{p-conj1.4} is true for this case. 

\bigskip

\noindent
{\bf [8-12]}\quad 
Let $Z_{12}:=X_\Sigma$ be the smooth complete toric threefold associated to 
the fan $\Sigma$ in $\mathbb{R}^3$ whose one-dimensional cones are 
generated by 
\[
 v_{1} = (1,0,0),  v_{2} = (0,1,0),  v_{3} = (0,0,1), v_{4} =
    (0,-1,-1), 
\]
\[
 v_{5} = (-1,0,-1),  v_{6} = (-2,-1,0),v_7 = (-1,-1,-1), v_8 = (-2,-1,-1),
\]
and the three-dimensional cones of $\Sigma$ are 
\[
\<v_1, v_2, v_3\>,  \<v_1, v_2, v_4\>,  \<v_1, v_3, v_6\>,  \<v_1, v_4, v_6 \>,
\<v_2, v_3, v_5\>,  \<v_2, v_4, v_5\>, 
\]
\[\<v_3, v_5, v_6\>,  \<v_4, v_5, v_7\>, 
\<v_4, v_6, v_7\>,  \<v_5, v_6, v_8\>,  \<v_5, v_7, v_8\>,  \<v_6, v_7, v_8\>.
\]
The configuration of cones of $\Sigma$ are as the following diagram: 
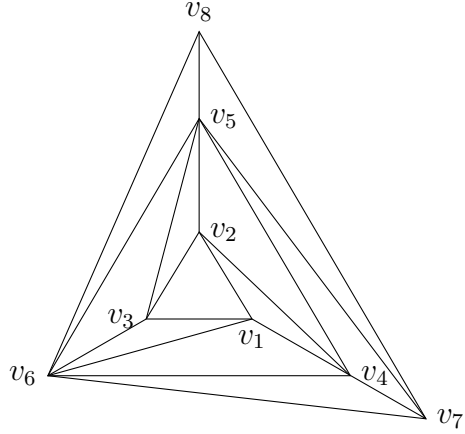
\begin{center}
\begin{tikzpicture}

\coordinate (B) at (4*0.5,3.8*0.5);
\coordinate (C) at (2.6*0.5,1.5*0.5);
\coordinate (A) at (5.4*0.5,1.5*0.5);
\coordinate (H) at (2,6.8*0.5+1.15);
\coordinate (F) at (0,0);
\coordinate (E) at (4*0.5,6.8*0.5);
\coordinate (D) at (8*0.5,0);
\coordinate (G) at (4+2*0.5,-1.14*0.5);

\draw (B)--(H);
\draw (A)--(G);
\draw (C)--(F);

\draw (C)--(E);
\draw (B)--(D);
\draw (A)--(F);

\draw (E)--(G);

\draw (A)--(B)--(C)--cycle;
\draw (D)--(E)--(F)--cycle;
\draw (F)--(G)--(H)--cycle;

\draw (A) node[below] {$v_1$};
\draw (B) node[right] {$v_2$};
\draw (C) node[left] {$v_3$};
\draw (D) node[right] {$v_4$};
\draw (E) node[right] {$v_5$};
\draw (F) node[left] {$v_6$};

\draw (G) node[right] {$v_7$};
\draw (H) node[above] {$v_8$};

\end{tikzpicture}
\end{center}
There is the extra three-dimensional cone 
$\langle v_6,v_7,v_8\rangle$. 

We obtain the sequence 
\[
Z_{12}\stackrel{\psi}{\dasharrow} X'\stackrel{\varphi_1}{\longrightarrow} Y_1\stackrel{\varphi_2}{\longrightarrow} Y_2
\]
of birational maps associated to the diagrams 
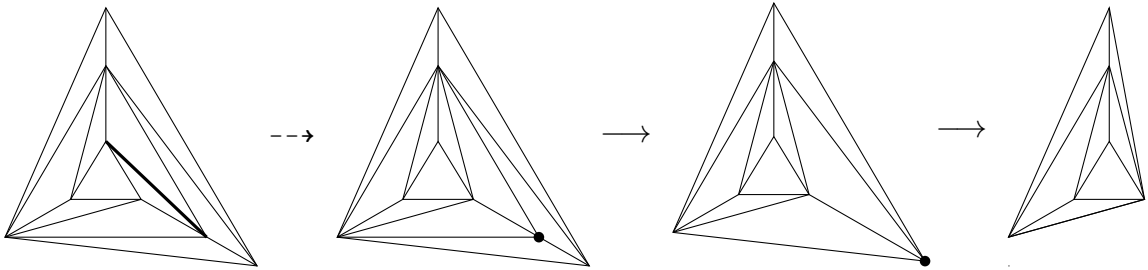
\begin{center}
\begin{tikzpicture}
\coordinate (B) at (4*0.5*2/3,3.8*0.5*2/3);
\coordinate (C) at (2.6*0.5*2/3,1.5*0.5*2/3);
\coordinate (A) at (5.4*0.5*2/3,1.5*0.5*2/3);
\coordinate (H) at (2*2/3,6.8*0.5*2/3+1.15*2/3);
\coordinate (F) at (0,0);
\coordinate (E) at (4*0.5*2/3,6.8*0.5*2/3);
\coordinate (D) at (8*0.5*2/3,0);
\coordinate (G) at (4*2/3+2*0.5*2/3,-1.14*0.5*2/3);

\draw (B)--(H);
\draw (A)--(G);
\draw (C)--(F);

\draw (C)--(E);
\draw [very thick] (B)--(D);
\draw (A)--(F);

\draw (E)--(G);

\draw (A)--(B)--(C)--cycle;
\draw (D)--(E)--(F)--cycle;
\draw (F)--(G)--(H)--cycle;

\draw (10/3,4/3) node[right] {$\dashrightarrow$};

\end{tikzpicture}
\begin{tikzpicture}
\coordinate (B) at (4*0.5*2/3,3.8*0.5*2/3);
\coordinate (C) at (2.6*0.5*2/3,1.5*0.5*2/3);
\coordinate (A) at (5.4*0.5*2/3,1.5*0.5*2/3);
\coordinate (H) at (2*2/3,6.8*0.5*2/3+1.15*2/3);
\coordinate (F) at (0,0);
\coordinate (E) at (4*0.5*2/3,6.8*0.5*2/3);
\coordinate (D) at (8*0.5*2/3,0);
\coordinate (G) at (4*2/3+2*0.5*2/3,-1.14*0.5*2/3);

\draw (B)--(H);
\draw (A)--(G);
\draw (C)--(F);

\draw (C)--(E);
\draw (A)--(E);
\draw (A)--(F);

\draw (E)--(G);

\draw (A)--(B)--(C)--cycle;
\draw (D)--(E)--(F)--cycle;
\draw (F)--(G)--(H)--cycle;

\fill[black] (D) circle (0.07);

\draw (10/3,4/3) node[right] {$\longrightarrow$};

\end{tikzpicture}
\begin{tikzpicture}
\coordinate (B) at (4*0.5*2/3,3.8*0.5*2/3);
\coordinate (C) at (2.6*0.5*2/3,1.5*0.5*2/3);
\coordinate (A) at (5.4*0.5*2/3,1.5*0.5*2/3);
\coordinate (H) at (2*2/3,6.8*0.5*2/3+1.15*2/3);
\coordinate (F) at (0,0);
\coordinate (E) at (4*0.5*2/3,6.8*0.5*2/3);
\coordinate (D) at (8*0.5*2/3,0);
\coordinate (G) at (4*2/3+2*0.5*2/3,-1.14*0.5*2/3);

\draw (B)--(H);
\draw (A)--(G);
\draw (C)--(F);

\draw (C)--(E);
\draw (A)--(E);
\draw (A)--(F);

\draw (E)--(G);

\draw (A)--(B)--(C)--cycle;
\draw (E)--(F);
\draw (F)--(G)--(H)--cycle;

\fill[black] (G) circle (0.07);

\draw (10/3,4/3) node[right] {$\longrightarrow$};

\end{tikzpicture}
\begin{tikzpicture}
\coordinate (B) at (4*0.5*2/3,3.8*0.5*2/3);
\coordinate (C) at (2.6*0.5*2/3,1.5*0.5*2/3);
\coordinate (A) at (5.4*0.5*2/3,1.5*0.5*2/3);
\coordinate (H) at (2*2/3,6.8*0.5*2/3+1.15*2/3);
\coordinate (F) at (0,0);
\coordinate (E) at (4*0.5*2/3,6.8*0.5*2/3);
\coordinate (D) at (8*0.5*2/3,0);
\coordinate (G) at (4*2/3+2*0.5*2/3,-1.14*0.5*2/3);

\draw (B)--(H);

\draw (C)--(F);

\draw (C)--(E);
\draw (A)--(E);
\draw (A)--(F);

\draw (A)--(B)--(C)--cycle;
\draw (E)--(F);
\draw (F)--(A)--(H)--cycle;

\draw [opacity=0] (0,-1.14/3)--(0,-1.14/3);

\end{tikzpicture}
\end{center}
of fans in $\mathbb{R}^3$. The rational map 
$\psi$ is the flop associated to the primitive relation 
\[
v_1+v_5=v_2+v_4,
\]
that is, the fan corresponding to $X'$ is obtained from $\Sigma$ 
by removing $\<v_2,v_4\>$ and by adding $\<v_1,v_5\>$. 
The morphisms 
$\varphi_1$ and $\varphi_2$ are associated to 
the star subdivisions along 
\[
v_4=v_1+v_7\mbox{ and }
v_7=v_1+v_8
\] 
of fans, respectively. 
Namely, $\psi$  is the flop along $V(\langle v_2,v_4\rangle)$, 
$\varphi_1$ is the blow-up of $Y_1$ along the curve $V(\langle v_1,v_7\rangle)$ and 
$\varphi_2$ is the blow-up of $Y_2$ along the curve $V(\langle v_1,v_8\rangle)$. 
Since $Y_2$ is smooth and its Picard number is three, 
$Y_2$, $Y_1$ and $X'$ are projective (see Remark \ref{p-rem1.6}). 
Conjecture \ref{p-conj1.4} is true for this case. 

\begin{rem}\label{p-rem4.5}
The variety $Z_{12}$ is studied in \cite[Example 1]{fujino-payne}, 
which is the first example of smooth complete 
toric threefolds with Picard number five and 
no non-trivial nef line bundles. 
In particular, it is non-projective. 
If we put $n=(-1, -1, -1)$, 
$n'=(1, 0, 0)$, and $n''=(0, 0, 1)$, 
then $Z_{12}$ coincides with [8-12] in \cite[p.79]{oda1}. Also, 
there is another description for $Z_{12}$ in \cite[Example 2.2.8]{fujino-foundations}. 
\end{rem}


\noindent
{\bf [8-13$'$]}\quad 
Let $a,b\in\mathbb{Z}$ and $Z'_{13}(a,b):=X_\Sigma$ 
the smooth complete toric threefold associated to 
the fan $\Sigma$ in $\mathbb{R}^3$ whose one-dimensional cones are 
generated by 
\[
v_1 = (-1,b,0),   v_2 = (0,-1,0),  v_3 = (1,-1,0), 
v_4 = (-1,0,-1), 
\]
\[
v_5 = (0,0,-1),  v_6 = (0,1,0), 
v_7 = (0,0,1),  v_8 = (1,0,a),
\]
and the three-dimensional cones of $\Sigma$ are 
\[
\<v_1, v_2, v_4\>,  \<v_2, v_3, v_4\>,  \<v_3, v_4, v_5\>,  \<v_4, v_5, v_6 \>,
\<v_1, v_4, v_6\>,  \<v_1, v_6, v_7\>, 
\]
\[
 \<v_1, v_2, v_7\>,  \<v_2, v_3, v_7\>, 
\<v_3, v_5, v_8\>,  \<v_5, v_6, v_8\>,  \<v_3, v_7, v_8\>,  \<v_6, v_7, v_8\>.
\]
The configuration of cones of $\Sigma$ are as the following diagram: 
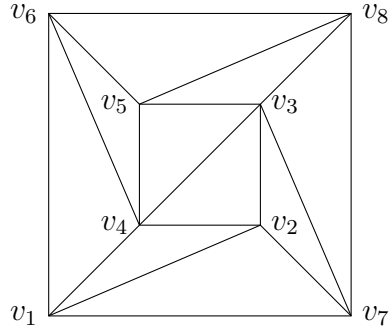
\begin{center}
\begin{tikzpicture}

\coordinate (E) at (1.2,1.2);
\coordinate (C) at (0,0);
\coordinate (B) at (4,0);
\coordinate (D) at (2.8,1.2);
\coordinate (F) at (2.8,2.8);
\coordinate (H) at (1.2,2.8);
\coordinate (G) at (0,4);
\coordinate (A) at (4,4);

\draw (E)--(F);

\draw (A)--(H);
\draw (B)--(F);
\draw (C)--(D);
\draw (E)--(G);

\draw (H)--(G);
\draw (F)--(A);
\draw (B)--(D);
\draw (C)--(E);

\draw (C)--(B)--(A)--(G)--cycle;
\draw (H)--(E)--(D)--(F)--cycle;

\draw (A) node[right] {$v_8$};
\draw (B) node[right] {$v_7$};
\draw (C) node[left] {$v_1$};
\draw (D) node[right] {$v_2$};
\draw (E) node[left] {$v_4$};
\draw (F) node[right] {$v_3$};

\draw (G) node[left] {$v_6$};
\draw (H) node[left] {$v_5$};

\end{tikzpicture}
\end{center}
In this case, there are two extra three-dimensional cones 
$\langle v_1,v_6,v_7\rangle$ and $\langle v_6,v_7,v_8\rangle$. 

We have primitive relations 
\[v_2+v_8=v_3+av_7 \quad \text{and}\quad  
v_3+v_6=av_5+v_8
\]
when $a>0$, and 
\[v_2+v_8=v_3+(-a)v_5 \quad \text{and}\quad  
v_3+v_6=(-a)v_7+v_8
\]
when $a<0$. 
Moreover, we have primitive relations 
\[v_1+v_5=v_4+bv_6 \quad \text{and}\quad  
v_4+v_7=v_1+bv_2
\]
when $b>0$, and 
\[v_1+v_5=(-b)v_2+v_4 \quad \text{and}\quad  
v_4+v_7=v_1+(-b)v_6
\]
when $b<0$. 
By the above primitive relations and Lemma \ref{p-lem2.9}, 
we can prove that 
there is no effective ample Cartier divisor on $Z'_{13}(a, b)$ when 
$ab\ne 0$. 
Therefore, $Z'_{13}(a, b)$ is not projective when $ab\ne 0$. 
When $a=0$, we have $v_3+v_6=v_8$. 
Therefore, we have a contraction 
$Z'_{13}(0, b)\to V_1$ by removing $v_8$. 
Similarly, when $b=0$, 
we have $v_4+v_7=v_1$. Then we have a contraction 
$Z'_{13}(a, 0)\to V_2$ by removing $v_1$. 
Note that $V_1$ and $V_2$ are smooth 
complete toric threefolds with Picard number four. 
They are not isomorphic to 
$W$ in Example \ref{p-ex4.1} since 
$v_2+v_6=v_5+v_7=0$. 
Therefore, $V_1$ and $V_2$ are both projective. 
This implies that $Z'_{13}
(a, b)$ is projective when $ab=0$. 

Let $Z'_{13}(a,b)\dashrightarrow X'$ be the flop 
along $V(\langle v_3,v_4\rangle)$, that is, there is a primitive 
relation 
\[
v_2+v_5=v_3+v_4.
\]
We can see that 
\[
X'=Z''_{13}(-1,a,b-1,-1)
\]
in [8-13$''$] below 
by the automorphism
\[
(x,y,z)\mapsto (x,x+y,z)
\]
of $\mathbb{R}^3$, and 
by the changing 
\[
1\mapsto8,2\mapsto 3,3\mapsto 2,4\mapsto 5,5\mapsto 4,8\mapsto 1
\]
of indices of $v_i$'s. Since Conjecture \ref{p-conj1.4} holds for $X'$ by the argument below, 
it holds for $Z'_{13}(a,b)$, too. 

\begin{rem}\label{p-rem4.6}
The variety $Z'_{13}(a,b)$ is studied in \cite[Example 3]{fujino-payne}. 
Our description here coincides with 
[8-13$'$] in \cite[p.79]{oda1} if 
we put $n=(1, -1, 0)$, $n'=(0, 1, 0)$, and $n''=(0, 0, 1)$. 
\end{rem}

\begin{rem}\label{p-rem4.7} The reader can find some 
related topics on [8-13$'$] in \cite[4.1.2]{bonavero}, 
where the toric variety 
$Z'_{13}(a\pm 1, b\pm 1)$ can be obtained from $Z'_{13}(a, b)$ by 
an elementary transformation (see \cite[Proposition 4]{bonavero}). 
\end{rem}


\noindent
{\bf [8-13$''$]}\quad 
Let $a,b,c,d\in\mathbb{Z}$ and $Z''_{13}(a,b,c,d):=X_\Sigma$ the smooth complete toric threefold associated to 
the fan $\Sigma$ in $\mathbb{R}^3$ whose one-dimensional cones are 
generated by 
\[
v_1 = (1,1,b),  v_2 = (1,0,0), 
v_3 = (0,-1,0), v_4 = (0,0,-1), 
\]
\[
v_5 = (-1,a,d),  v_6 = (0,1,0), 
v_7 = (0,0,1),  v_8 = (-1,c,d+1),
\]
and the three-dimensional cones of $\Sigma$ are 
\[
\<v_1, v_2, v_4\>,  \<v_2, v_3, v_4\>,  \<v_3, v_4, v_5\>,  \<v_4, v_5, v_6 \>,
\<v_1, v_4, v_6\>,  \<v_1, v_6, v_7\>, 
\]
\[\<v_1, v_2, v_7\>,  \<v_2, v_3, v_7\>, 
\<v_3, v_5, v_8\>,  \<v_5, v_6, v_8\>,  \<v_3, v_7, v_8\>,  \<v_6, v_7, v_8\>.
\]
The configuration of cones of $\Sigma$ are as the following diagram: 
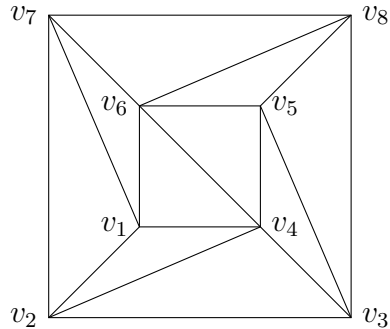
\begin{center}
\begin{tikzpicture}

\coordinate (A) at (1.2,1.2);
\coordinate (B) at (0,0);
\coordinate (C) at (4,0);
\coordinate (D) at (2.8,1.2);
\coordinate (E) at (2.8,2.8);
\coordinate (F) at (1.2,2.8);
\coordinate (G) at (0,4);
\coordinate (H) at (4,4);

\draw (D)--(F);

\draw (F)--(G);
\draw (E)--(H);
\draw (D)--(C);
\draw (A)--(B);

\draw (A)--(G);
\draw (F)--(H);
\draw (E)--(C);
\draw (D)--(B);

\draw (B)--(C)--(H)--(G)--cycle;
\draw (A)--(D)--(E)--(F)--cycle;

\draw (A) node[left] {$v_1$};
\draw (B) node[left] {$v_2$};
\draw (C) node[right] {$v_3$};
\draw (D) node[right] {$v_4$};
\draw (E) node[right] {$v_5$};
\draw (F) node[left] {$v_6$};

\draw (G) node[left] {$v_7$};
\draw (H) node[right] {$v_8$};

\end{tikzpicture}
\end{center}
In this case, there are two extra three-dimensional cones 
$\langle v_2,v_3,v_7\rangle$ and $\langle v_3,v_7,v_8\rangle$. 

We have primitive relations 
\[v_1+v_3=v_2+bv_7 \quad \text{and}\quad  
v_2+v_6=v_1+bv_4
\]
when $b>0$, and 
\[v_1+v_3=v_2+(-b)v_4 \quad \text{and}\quad  
v_2+v_6=v_1+(-b)v_7
\]
when $b<0$. 
Moreover, we have primitive relations 
\[v_4+v_8=v_5+(a-c)v_3 \quad \text{and}\quad  
v_5+v_7=v_8+(a-c)v_6
\]
when $a>c$, and 
\[v_4+v_8=v_5+(c-a)v_6 \quad \text{and}\quad  
v_5+v_7=v_8+(c-a)v_3
\]
when $a<c$.
By the above primitive relations and Lemma \ref{p-lem2.9}, we can prove that 
there exists no effective ample Cartier 
divisor on $Z''_{13}(a, b, c, d)$ if $a\ne c$ and $b\ne 0$. 
This means that $Z''_{13}(a, b, c, d)$ is not projective 
when $a\ne c$ and $b\ne 0$. 

Since $v_3+v_6=v_4+v_7=0$ holds, 
there exists a morphism $Z''_{13}(a,b,c,d)\to \mathbb{P}^1$ 
with the general fiber $\mathbb{P}^1\times\mathbb{P}^1$ 
associated to 
the first projection $\mathbb{R}^3\ni (x,y,z)\mapsto x\in\mathbb{R}$. 
The followings are the pictures for two sub-fans $\Sigma_+\subset\Sigma$ in 
$\{(x,y,z)\,|\,x\ge 0\}\subset\mathbb{R}^3$ and 
$\Sigma_-\subset\Sigma$ in 
$\{(x,y,z)\,|\,x\le 0\}\subset\mathbb{R}^3$: 
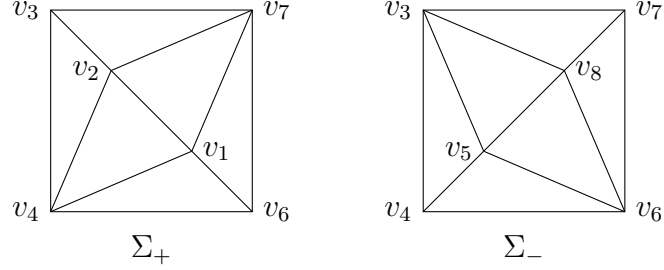
\begin{center}
\begin{tikzpicture}

\coordinate (C) at (0,4*2/3);
\coordinate (D) at (0,0);

\coordinate (A) at (2.8*2/3,1.2*2/3);

\coordinate (F) at (4*2/3,0);
\coordinate (G) at (4*2/3,4*2/3);

\coordinate (B) at (1.2*2/3,2.8*2/3);

\draw (D)--(A);
\draw (D)--(B);

\draw (G)--(A);
\draw (G)--(B);

\draw (C)--(B);

\draw (A)--(B);
\draw (A)--(F);

\draw (C)--(D)--(F)--(G)--cycle;

\draw (C) node[left] {$v_3$};
\draw (D) node[left] {$v_4$};
\draw (G) node[right] {$v_7$};
\draw (F) node[right] {$v_6$};

\draw (A) node[right] {$v_1$};
\draw (B) node[left] {$v_2$};

\draw (2*2/3,-0.5) node {$\Sigma_+$};

\end{tikzpicture}
\quad \quad
\begin{tikzpicture}

\coordinate (C) at (0,4*2/3);
\coordinate (D) at (0,0);

\coordinate (E) at (1.2*2/3,1.2*2/3);

\coordinate (F) at (4*2/3,0);
\coordinate (G) at (4*2/3,4*2/3);

\coordinate (H) at (2.8*2/3,2.8*2/3);

\draw (C)--(E);
\draw (C)--(H);

\draw (F)--(E);
\draw (F)--(H);

\draw (D)--(E);

\draw (E)--(H);
\draw (H)--(G);

\draw (C)--(D)--(F)--(G)--cycle;

\draw (C) node[left] {$v_3$};
\draw (D) node[left] {$v_4$};
\draw (E) node[left] {$v_5$};
\draw (F) node[right] {$v_6$};

\draw (G) node[right] {$v_7$};
\draw (H) node[right] {$v_8$};

\draw (2*2/3,-0.5) node {$\Sigma_-$};

\end{tikzpicture}
\end{center}
For $\Sigma_+$, we have the primitive relation:
\[
v_2+v_6=
\left\{
\begin{array}{ccc}
v_1 &\cdots & b= 0\\
v_1+bv_4 & \cdots & b>0\\
v_1+(-b)v_7 & \cdots & b<0
\end{array}
\right.
\]
We remark that $Z''_{13}(a,0,c,d)$ 
can be blown-down to a smooth 
threefold $V$. It is not difficult to 
see that $V$ is not isomorphic to 
$W$ in Example \ref{p-ex4.1} since $v_3+v_6=v_4+v_7=0$. 
Hence $V$ is projective. 
Therefore, $Z''_{13}(a,0,c,d)$ itself is projective. 
So, let $b>0$. Then, after the 
anti-flip (flop if $b=1$) $Z''_{13}(a,b,c,d)\dashrightarrow X'$ 
along $V(\langle v_1,v_4\rangle)$, 
there exists the relation $v_2+v_6+bv_7=v_1$. 
Thus, we obtain the divisorial 
contraction $X'\to Y$ by removing the one-dimensional cone 
generated by $v_1$. 
Note that $Y$ is smooth. 
The case where $b<0$ is completely similar. 

On the other hand, 
for $\Sigma_-$, we have the primitive relation:
\[
v_5+v_7=
\left\{
\begin{array}{ccc}
v_8 &\cdots & a= c\\
v_8+(a-c)v_6 & \cdots & a>c\\
v_8+(c-a)v_3 & \cdots & a<c
\end{array}
\right.
\]
So, we can do the same operation as $\Sigma_+$ for 
the other side $\Sigma_-$ independently. 
It can be summarized as follows. 
If $a-c=0$, then we can remove $v_8$, which is a 
divisorial contraction to a smooth threefold. Then we can check that 
$Z''_{13}(a, b, a, d)$ is projective as in the case where $b=0$. 
If $a-c=\pm 1$, then 
we can remove $v_8$ after a flop. If $a-c\ne 0, \pm 1$, 
then we can remove $v_8$ after an anti-flip. 

Thus, if $b\ne 0$ and $a-c\ne 0$, then $Z''_{13}$ is not 
projective and we have the sequence
\[
Z''_{13}\stackrel{\psi'}{\dashrightarrow} X'
\stackrel{\psi''}{\dashrightarrow} X''\stackrel{\varphi_1}{\longrightarrow}
 Y_1\stackrel{\varphi_2}{\longrightarrow} Y_2
\]  
of complete toric threefolds, where $\psi',\psi''$ are anti-flips (or flops) and 
$\varphi_1,\varphi_2$ are divisorial contractions which contract 
a divisor to a smooth point. 
Since $Y_2$ is a smooth complete toric threefold of 
Picard number three, $Y_2$, $Y_1$ and $X''$ are projective. 
So, Conjecture \ref{p-conj1.4} holds for $Z''_{13}$. 

\bigskip

\noindent
{\bf [8-14$'$]}\quad 
Let $a\in\mathbb{Z}$ and $Z'_{14}(a):=X_\Sigma$ the smooth complete toric threefold associated to 
the fan $\Sigma$ in $\mathbb{R}^3$ whose one-dimensional cones are 
generated by 
\[
v_1 = (0,0,-1),  v_2 = (1,0,0), 
v_3 = (0,1,0), v_4 = (-1,-1,a), 
\]
\[
v_5 = (-1,-1,a+1),  v_6 = (1,0,1), 
v_7 = (0,0,1),  v_8 = (0,1,1),
\]
and the three-dimensional cones of $\Sigma$ are 
\[
\<v_1, v_2, v_3\>,  \<v_1, v_3, v_4\>,  \<v_1, v_2, v_4\>,  \<v_3, v_4, v_5\>, 
\<v_4, v_5, v_6\>,  \<v_2, v_4, v_6\>, 
\]
\[
 \<v_5, v_6, v_7\>,  \<v_2, v_3, v_8\>, 
\<v_3, v_5, v_8\>,  \<v_5, v_7, v_8\>,  \<v_6, v_7, v_8\>,  \<v_2, v_6, v_8\>.
\]
The configuration of cones of $\Sigma$ are as the following diagram: 
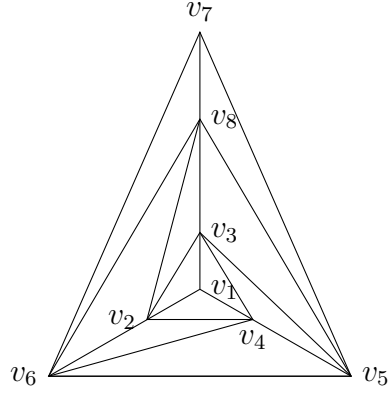
\begin{center}
\begin{tikzpicture}

\coordinate (C) at (4*0.5,3.8*0.5);
\coordinate (B) at (2.6*0.5,1.5*0.5);
\coordinate (D) at (5.4*0.5,1.5*0.5);
\coordinate (G) at (2,6.8*0.5+1.15);
\coordinate (F) at (0,0);
\coordinate (H) at (4*0.5,6.8*0.5);
\coordinate (E) at (8*0.5,0);
\coordinate (A) at (2,2.3*0.5);

\draw (A)--(G);
\draw (A)--(E);
\draw (A)--(F);

\draw (B)--(H);
\draw (C)--(E);
\draw (D)--(F);

\draw (B)--(C)--(D)--cycle;
\draw (E)--(F)--(H)--cycle;
\draw (E)--(F)--(G)--cycle;

\draw (A) node[right] {$v_1$};
\draw (B) node[left] {$v_2$};
\draw (C) node[right] {$v_3$};
\draw (D) node[below] {$v_4$};
\draw (E) node[right] {$v_5$};
\draw (F) node[left] {$v_6$};

\draw (G) node[above] {$v_7$};
\draw (H) node[right] {$v_8$};

\end{tikzpicture}
\end{center}
There is the extra three-dimensional cone 
$\langle v_5,v_6,v_7\rangle$. 

We have primitive relations 
\[v_2+v_5=v_4+v_6, \quad 
v_3+v_6=v_2+v_8, \quad \text{and} 
\quad
v_4+v_8=v_3+v_5. 
\] 
This implies that $Z'_{14}(a)$ is non-projective 
by the above primitive relations and Lemma \ref{p-lem2.9}. 

We obtain the sequence 
\[
Z'_{14}(a)\stackrel{\psi}{\dasharrow} X'\stackrel{\varphi_1}{\longrightarrow} Y_1\stackrel{\varphi_2}{\longrightarrow} Y_2
\]
of birational maps associated to the diagrams 
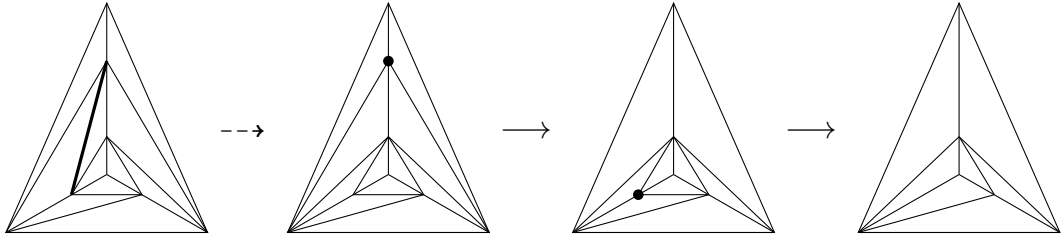
\begin{center}
\begin{tikzpicture}
\coordinate (C) at (4*0.5*2/3,3.8*0.5*2/3);
\coordinate (B) at (2.6*0.5*2/3,1.5*0.5*2/3);
\coordinate (D) at (5.4*0.5*2/3,1.5*0.5*2/3);
\coordinate (G) at (2*2/3,6.8*0.5*2/3+1.15*2/3);
\coordinate (F) at (0,0);
\coordinate (H) at (4*0.5*2/3,6.8*0.5*2/3);
\coordinate (E) at (8*0.5*2/3,0);
\coordinate (A) at (2*2/3,2.3*0.5*2/3);

\draw (A)--(G);
\draw (A)--(E);
\draw (A)--(F);

\draw [very thick] (B)--(H);
\draw (C)--(E);
\draw (D)--(F);

\draw (B)--(C)--(D)--cycle;
\draw (E)--(F)--(H)--cycle;
\draw (E)--(F)--(G)--cycle;

\draw (8*0.5*2/3,4/3) node[right] {$\dashrightarrow$};

\end{tikzpicture}
\begin{tikzpicture}
\coordinate (C) at (4*0.5*2/3,3.8*0.5*2/3);
\coordinate (B) at (2.6*0.5*2/3,1.5*0.5*2/3);
\coordinate (D) at (5.4*0.5*2/3,1.5*0.5*2/3);
\coordinate (G) at (2*2/3,6.8*0.5*2/3+1.15*2/3);
\coordinate (F) at (0,0);
\coordinate (H) at (4*0.5*2/3,6.8*0.5*2/3);
\coordinate (E) at (8*0.5*2/3,0);
\coordinate (A) at (2*2/3,2.3*0.5*2/3);

\draw (A)--(G);
\draw (A)--(E);
\draw (A)--(F);

\draw (C)--(F);
\draw (C)--(E);
\draw (D)--(F);

\draw (B)--(C)--(D)--cycle;
\draw (E)--(F)--(H)--cycle;
\draw (E)--(F)--(G)--cycle;

\fill[black] (H) circle (0.07);

\draw (8*0.5*2/3,4/3) node[right] {$\longrightarrow$};

\end{tikzpicture}
\begin{tikzpicture}
\coordinate (C) at (4*0.5*2/3,3.8*0.5*2/3);
\coordinate (B) at (2.6*0.5*2/3,1.5*0.5*2/3);
\coordinate (D) at (5.4*0.5*2/3,1.5*0.5*2/3);
\coordinate (G) at (2*2/3,6.8*0.5*2/3+1.15*2/3);
\coordinate (F) at (0,0);
\coordinate (H) at (4*0.5*2/3,6.8*0.5*2/3);
\coordinate (E) at (8*0.5*2/3,0);
\coordinate (A) at (2*2/3,2.3*0.5*2/3);

\draw (A)--(G);
\draw (A)--(E);
\draw (A)--(F);

\draw (C)--(F);
\draw (C)--(E);
\draw (D)--(F);

\draw (B)--(C)--(D)--cycle;
\draw (E)--(F);
\draw (E)--(F)--(G)--cycle;

\fill[black] (B) circle (0.07);

\draw (8*0.5*2/3,4/3) node[right] {$\longrightarrow$};

\end{tikzpicture}
\begin{tikzpicture}
\coordinate (C) at (4*0.5*2/3,3.8*0.5*2/3);
\coordinate (B) at (2.6*0.5*2/3,1.5*0.5*2/3);
\coordinate (D) at (5.4*0.5*2/3,1.5*0.5*2/3);
\coordinate (G) at (2*2/3,6.8*0.5*2/3+1.15*2/3);
\coordinate (F) at (0,0);
\coordinate (H) at (4*0.5*2/3,6.8*0.5*2/3);
\coordinate (E) at (8*0.5*2/3,0);
\coordinate (A) at (2*2/3,2.3*0.5*2/3);

\draw (A)--(G);
\draw (A)--(E);
\draw (A)--(F);

\draw (C)--(F);
\draw (C)--(E);
\draw (D)--(F);

\draw (C)--(D);
\draw (E)--(F);
\draw (E)--(F)--(G)--cycle;

\end{tikzpicture}
\end{center}
of fans in $\mathbb{R}^3$. The rational map 
$\psi$ is the flop associated to the primitive relation 
\[
v_3+v_6=v_2+v_8,
\]
that is, the fan corresponding to $X'$ is obtained from $\Sigma$ 
by removing $\<v_2,v_8\>$ and by adding $\<v_3,v_6\>$. 
The morphisms 
$\varphi_1$ and $\varphi_2$ are associated to 
the star subdivisions along 
\[
v_8=v_3+v_7\mbox{ and }
v_2=v_1+v_6
\] 
of fans, respectively. 
Namely, $\psi$  is the flop along $V(\langle v_2,v_8\rangle)$, 
$\varphi_1$ is the blow-up of $Y_1$ along the curve $V(\langle v_3,v_7\rangle)$ and 
$\varphi_2$ is the blow-up of $Y_2$ along the curve $V(\langle v_1,v_6\rangle)$. 
Since $Y_2$ is smooth and its Picard number is three, 
$Y_2$, $Y_1$ and $X'$ are projective (see Remark \ref{p-rem1.6}). 
Conjecture \ref{p-conj1.4} is true for this case. 

\bigskip

\noindent
{\bf [8-14$''$]}\quad 
Let $a,b\in\mathbb{Z}$ and $Z''_{14}(a,b):=X_\Sigma$ the smooth complete toric threefold associated to 
the fan $\Sigma$ in $\mathbb{R}^3$ whose one-dimensional cones are 
generated by 
\[
v_1 = (-1,a,b),  v_2 = (0,1,0), 
v_3 = (0,-1,-1), v_4 = (0,0,1), 
\]
\[
v_5 = (1,0,1),  v_6 = (1,1,0), 
v_7 = (1,0,0),  v_8 = (1,-1,-1),
\]
and the three-dimensional cones of $\Sigma$ are 
\[
\<v_1, v_2, v_3\>,  \<v_1, v_3, v_4\>,  \<v_1, v_2, v_4\>,  \<v_3, v_4, v_5\>, 
\<v_4, v_5, v_6\>,  \<v_2, v_4, v_6\>, 
\]
\[
 \<v_5, v_6, v_7\>,  \<v_2, v_3, v_8\>, 
\<v_3, v_5, v_8\>,  \<v_5, v_7, v_8\>,  \<v_6, v_7, v_8\>,  \<v_2, v_6, v_8\>.
\]
The configuration of cones of $\Sigma$ are as the following diagram: 
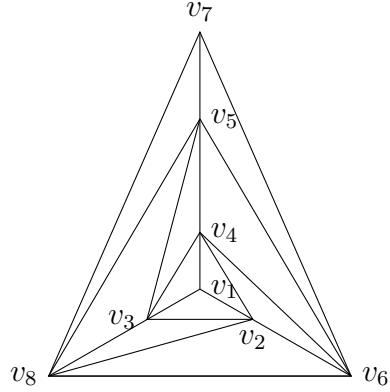
\begin{center}
\begin{tikzpicture}

\coordinate (D) at (4*0.5,3.8*0.5);
\coordinate (C) at (2.6*0.5,1.5*0.5);
\coordinate (B) at (5.4*0.5,1.5*0.5);
\coordinate (G) at (2,6.8*0.5+1.15);
\coordinate (H) at (0,0);
\coordinate (E) at (4*0.5,6.8*0.5);
\coordinate (F) at (8*0.5,0);
\coordinate (A) at (2,2.3*0.5);

\draw (A)--(F);
\draw (A)--(G);
\draw (A)--(H);

\draw (B)--(H);
\draw (C)--(E);
\draw (D)--(F);

\draw (B)--(C)--(D)--cycle;
\draw (E)--(F)--(H)--cycle;
\draw (F)--(G)--(H)--cycle;

\draw (A) node[right] {$v_1$};
\draw (B) node[below] {$v_2$};
\draw (C) node[left] {$v_3$};
\draw (D) node[right] {$v_4$};
\draw (E) node[right] {$v_5$};
\draw (F) node[right] {$v_6$};

\draw (G) node[above] {$v_7$};
\draw (H) node[left] {$v_8$};

\end{tikzpicture}
\end{center}
There is the extra three-dimensional cone 
$\langle v_6,v_7,v_8\rangle$. 

We have primitive relations 
\[v_2+v_5=v_4+v_6, \quad 
v_3+v_6=v_2+v_8, \quad \text{and} 
\quad
v_4+v_8=v_3+v_5. 
\] 
Hence $Z''_{14}(a, b)$ is not projective 
by the above primitive relations and Lemma \ref{p-lem2.9}. 

We obtain the sequence 
\[
Z''_{14}(a,b)\stackrel{\psi}{\dasharrow} X'\stackrel{\varphi_1}{\longrightarrow} Y_1\stackrel{\varphi_2}{\longrightarrow} Y_2
\]
of birational maps associated to the diagrams 
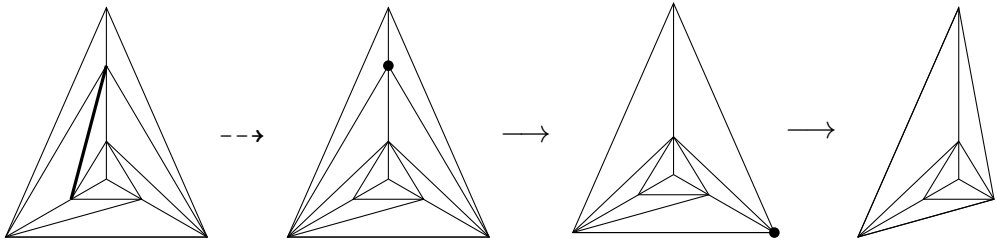
\begin{center}
\begin{tikzpicture}
\coordinate (D) at (4*0.5*2/3,3.8*0.5*2/3);
\coordinate (C) at (2.6*0.5*2/3,1.5*0.5*2/3);
\coordinate (B) at (5.4*0.5*2/3,1.5*0.5*2/3);
\coordinate (G) at (2*2/3,6.8*0.5*2/3+1.15*2/3);
\coordinate (H) at (0,0);
\coordinate (E) at (4*0.5*2/3,6.8*0.5*2/3);
\coordinate (F) at (8*0.5*2/3,0);
\coordinate (A) at (2*2/3,2.3*0.5*2/3);

\draw (A)--(F);
\draw (A)--(G);
\draw (A)--(H);

\draw (B)--(H);
\draw [very thick] (C)--(E);
\draw (D)--(F);

\draw (B)--(C)--(D)--cycle;
\draw (E)--(F)--(H)--cycle;
\draw (F)--(G)--(H)--cycle;

\draw (8*0.5*2/3,4/3) node[right] {$\dashrightarrow$};

\end{tikzpicture}
\begin{tikzpicture}
\coordinate (D) at (4*0.5*2/3,3.8*0.5*2/3);
\coordinate (C) at (2.6*0.5*2/3,1.5*0.5*2/3);
\coordinate (B) at (5.4*0.5*2/3,1.5*0.5*2/3);
\coordinate (G) at (2*2/3,6.8*0.5*2/3+1.15*2/3);
\coordinate (H) at (0,0);
\coordinate (E) at (4*0.5*2/3,6.8*0.5*2/3);
\coordinate (F) at (8*0.5*2/3,0);
\coordinate (A) at (2*2/3,2.3*0.5*2/3);

\draw (A)--(F);
\draw (A)--(G);
\draw (A)--(H);

\draw (B)--(H);
\draw (D)--(H);
\draw (D)--(F);

\draw (B)--(C)--(D)--cycle;
\draw (E)--(F)--(H)--cycle;
\draw (F)--(G)--(H)--cycle;

\fill[black] (E) circle (0.07);

\draw (8*0.5*2/3,4/3) node[right] {$\longrightarrow$};

\end{tikzpicture}
\begin{tikzpicture}
\coordinate (D) at (4*0.5*2/3,3.8*0.5*2/3);
\coordinate (C) at (2.6*0.5*2/3,1.5*0.5*2/3);
\coordinate (B) at (5.4*0.5*2/3,1.5*0.5*2/3);
\coordinate (G) at (2*2/3,6.8*0.5*2/3+1.15*2/3);
\coordinate (H) at (0,0);
\coordinate (E) at (4*0.5*2/3,6.8*0.5*2/3);
\coordinate (F) at (8*0.5*2/3,0);
\coordinate (A) at (2*2/3,2.3*0.5*2/3);

\draw (A)--(F);
\draw (A)--(G);
\draw (A)--(H);

\draw (B)--(H);
\draw (D)--(H);
\draw (D)--(F);

\draw (B)--(C)--(D)--cycle;
\draw (F)--(H);
\draw (F)--(G)--(H)--cycle;

\fill[black] (F) circle (0.07);

\draw (8*0.5*2/3,4/3) node[right] {$\longrightarrow$};

\end{tikzpicture}
\begin{tikzpicture}
\coordinate (D) at (4*0.5*2/3,3.8*0.5*2/3);
\coordinate (C) at (2.6*0.5*2/3,1.5*0.5*2/3);
\coordinate (B) at (5.4*0.5*2/3,1.5*0.5*2/3);
\coordinate (G) at (2*2/3,6.8*0.5*2/3+1.15*2/3);
\coordinate (H) at (0,0);
\coordinate (E) at (4*0.5*2/3,6.8*0.5*2/3);
\coordinate (F) at (8*0.5*2/3,0);
\coordinate (A) at (2*2/3,2.3*0.5*2/3);

\draw (A)--(B);
\draw (A)--(G);
\draw (A)--(H);

\draw (B)--(H);
\draw (D)--(H);

\draw (G)--(H);

\draw (B)--(C)--(D)--cycle;

\draw (B)--(G)--(H)--cycle;

\end{tikzpicture}
\end{center}
of fans in $\mathbb{R}^3$. The rational map 
$\psi$ is the flop associated to the primitive relation 
\[
v_4+v_8=v_3+v_5,
\]
that is, the fan corresponding to $X'$ is obtained from $\Sigma$ 
by removing $\<v_3,v_5\>$ and by adding $\<v_4,v_8\>$. 
The morphisms 
$\varphi_1$ and $\varphi_2$ are associated to 
the star subdivisions along 
\[
v_5=v_4+v_7\mbox{ and }
v_6=v_2+v_7
\] 
of fans, respectively. 
Namely, $\psi$  is the flop along $V(\langle v_3,v_5\rangle)$, 
$\varphi_1$ is the blow-up of $Y_1$ along the curve $V(\langle v_4,v_7\rangle)$ and 
$\varphi_2$ is the blow-up of $Y_2$ along the curve $V(\langle v_2,v_7\rangle)$. 
Since $Y_2$ is smooth and its Picard number is three, 
$Y_2$, $Y_1$ and $X'$ are projective (see Remark \ref{p-rem1.6}). 
Conjecture \ref{p-conj1.4} is true for this case. 

\begin{rem}\label{p-rem4.8}
We have to check the remaining case, that is, 
the blow-ups of the smooth non-projective complete toric 
threefold $W$ of Picard number four 
in Example \ref{p-ex4.1}. 
However, one can easily find a flop that makes the variety projective in this case. 
In fact, at least one of the three flopping curves on $W$ is preserved by any blow-up. 
\end{rem}

Thus, we have the following:

\begin{thm}\label{p-thm4.9}
Conjecture \ref{p-conj1.4} is true for any smooth complete toric threefold of 
Picard number at most five. 
\end{thm}

\section{Proof of Theorem \ref{p-thm1.7}}\label{p-sec5}

In this final section, we prove Theorem \ref{p-thm1.7}. 
More precisely, we explicitly construct a smooth 
non-projective complete toric threefold satisfying the 
property in Theorem \ref{p-thm1.7}. 
In the following, we will use the notation in [8-13$''$]. 
Put $v_i=(x_i,y_i,z_i)$ for $1\le i\le 8$. 

\begin{lem}\label{p-lem5.1}
Let $\Delta$ be a non-singular complete fan whose one-dimensional 
cones are $\<v_1\>, \ldots, \<v_8\>$ in {\em{[8-13$''$]}}. 
Suppose that the following conditions hold:
\begin{enumerate}
\item[\em{(1)}]  
$a,c,d\not\in\{-2,-1,0,1\}$. 
\item[\em{(2)}] 
$c>a+1$. 
\item[\em{(3)}] 
For any $2\le i<j\le 8$, we have 
\[
b>
\left|
\pm 1
-
\begin{vmatrix}
y_i & z_i \\
y_j & z_j
\end{vmatrix}
+
\begin{vmatrix}
x_i & z_i \\
x_j & z_j
\end{vmatrix}
\right|.
\]
\end{enumerate}
Then $\Delta=\Sigma$ holds, where $\Sigma$ is the fan given in {\em{[8-13$''$]}}. 
\end{lem}

\begin{rem}\label{p-rem5.2}
The condition $(3)$ is equivalent to the following explicit condition: 
\[
(3)'\ b>\max
\left\{
2,|a|,|a+2|,|c|,|c+2|,|d-1|,|d+2|,|ad+a-cd|,|ad+a-cd+2|
\right\}. 
\]
This can be done by simply calculating the condition $(3)$ for any 
$2\le i<j\le 8$. 
\end{rem}

\begin{proof}[Proof of Lemma \ref{p-lem5.1}]
If $\<v_1,v_i,v_j\>$ is a {\em non-singular} three-dimensional cone for some 
$2\le i<j\le 8$, then 
we have 
\[
\begin{vmatrix}
v_1 \\ v_i \\ v_j
\end{vmatrix}
=
\begin{vmatrix}
y_i & z_i \\
y_j & z_j
\end{vmatrix}
-
\begin{vmatrix}
x_i & z_i \\
x_j & z_j
\end{vmatrix}
+b
\begin{vmatrix}
x_i & y_i \\
x_j & y_j
\end{vmatrix}
=\pm 1
\]
\[
\Longleftrightarrow\quad 
b
\begin{vmatrix}
x_i & y_i \\
x_j & y_j
\end{vmatrix}
=
\pm 1
-
\begin{vmatrix}
y_i & z_i \\
y_j & z_j
\end{vmatrix}
+
\begin{vmatrix}
x_i & z_i \\
x_j & z_j
\end{vmatrix}. 
\]
If $x_iy_j-y_ix_j\neq 0$, then the inequality 
\[
|b|
\le
\left|
\pm 1
-
\begin{vmatrix}
y_i & z_i \\
y_j & z_j
\end{vmatrix}
+
\begin{vmatrix}
x_i & z_i \\
x_j & z_j
\end{vmatrix}
\right|
\]
contradicts the condition $(3)$. 
Therefore, we have 
\[
\begin{vmatrix}
x_i & y_i \\
x_j & y_j
\end{vmatrix}=
0. 
\] 
Thus, 
by easy calculations, 
either $4\in\{i,j\}$, $7\in\{i,j\}$ or $(i,j)=(3,6)$ holds 
by the assumptions 
$a, c\neq 0$ and $a\neq c$. 
Obviously, $\<v_1,v_3,v_6\>$ is not a three-dimensional cone. 
Since 
\[
\begin{vmatrix}
v_1 \\ v_4 \\ v_5
\end{vmatrix}
=a+1\neq\pm1,\ 
\begin{vmatrix}
v_1 \\ v_7 \\ v_5
\end{vmatrix}
=-(a+1)\neq\pm1,\ 
\begin{vmatrix}
v_1 \\ v_4 \\ v_8
\end{vmatrix}
=c+1\neq\pm1\mbox{ and }
\begin{vmatrix}
v_1 \\ v_7 \\ v_8
\end{vmatrix}
=-(c+1)\neq\pm1 
\]
hold by the assumptions 
$a\neq 0,-2$ and $c\neq 0,-2$, 
the remaining possibilities for a three-dimensional cone 
which contains $\<v_1\>$ as a face are
\[
\langle v_1, v_2, v_7\rangle,\ 
\langle v_1, v_3, v_7\rangle,\ 
\langle v_1, v_6, v_7\rangle,\ 
\langle v_1, v_2, v_4\rangle,\ 
\langle v_1, v_3, v_4\rangle\mbox{ and }
\langle v_1, v_4, v_6\rangle. 
\]
We remark that these six cones are contained in the half space defined 
by $x\ge 0$. 
Among them, we note that 
\[
\langle v_1, v_2, v_4\rangle,\ 
\langle v_1, v_3, v_4\rangle,\ 
\langle v_1, v_2, v_7\rangle\mbox{ and }
\langle v_1, v_3, v_7\rangle
\]
are contained in the half space 
defined by $x\geq y$, while 
$\langle v_1\rangle$, $\langle v_4\rangle$ and $\langle v_7\rangle$ 
are contained in the hyperplane defined by $x=y$. 
Therefore, 
$\langle v_1, v_4, v_6\rangle$ and 
$\langle v_1, v_6, v_7\rangle$ must be in $\Delta$. 
Since the relation $v_1+v_3+bv_4=v_2$ tells us that 
$\<v_1,v_3,v_4\>$ contains $\<v_2\>$ inside 
by the assumption $b>0$, 
the cone $\langle v_1, v_3, v_4\rangle$ is not contained in $\Delta$. 
Hence, we have $\langle v_1, v_2, v_4\rangle \in \Delta$. 
Suppose that $\<v_2,v_4,v_i\>\in\Delta$ is the another three-dimensional cone 
which contains $\<v_2,v_4\>$. Then, 
\[
\begin{vmatrix}
v_2 \\ v_4 \\ v_i
\end{vmatrix}
=y_i=\pm 1
\]
says that $i=3$ by the assumptions $a\neq\pm 1$ and $c\neq\pm 1$. 
Thus, $\langle v_2, v_3, v_4\rangle \in \Delta$. 
Moreover, suppose that $\<v_2,v_3,v_j\>\in\Delta$ is the another three-dimensional cone 
which contains $\<v_2,v_3\>$. Then, the equality 
\[
\begin{vmatrix}
v_2 \\ v_3 \\ v_j
\end{vmatrix}
=-z_j=\pm 1
\]
implies that $j=7$ by the assumptions $b\neq\pm 1$ and $d\neq-2,-1,0,1$. 
Then we obtain that 
$\langle v_2, v_3, v_7\rangle$. Also, one can easily see 
$\langle v_1, v_2, v_7\rangle$ must be in $\Delta$. 
Hence $\Delta$ coincides with $\Sigma$ in the half space defined by $x\geq 0$. 

From now, we see the half space defined by $x\leq 0$. 
By the assumption $c>a$, the relation 
\[
(c-a)v_3+v_4+v_8=v_5
\]
says that 
$\langle v_3, v_4, v_8\rangle$ contains $\langle v_5\rangle$ inside, and 
$\langle v_3, v_4, v_8\rangle \not \in \Delta$. 
Hence we have $\langle v_3, v_4, v_5\rangle \in \Delta$. 
Since $v_4+v_7=0$, the another three-dimensional cone in $\Delta$ 
which contains $\<v_4,v_5\>$ is either 
$\langle v_4, v_5, v_6\rangle$ or $\langle v_4, v_5, v_8\rangle$. 
However, 
\[
\begin{vmatrix}
v_4 \\ v_5 \\ v_8
\end{vmatrix}
=c-a\neq\pm 1
\]
by the assumption $c>a+ 1$. 
This implies $\langle v_4, v_5, v_6\rangle \in \Delta$. 
Finally, since also 
\[
\begin{vmatrix}
v_5 \\ v_7 \\ v_8
\end{vmatrix}
=c-a\neq\pm 1
\]
holds, we have $\langle v_5, v_7, v_8\rangle \not\in \Delta$. 
The remaining possibilities for a three-dimensional cone 
which contains $\<v_8\>$ as a face are 
\[
\langle v_3, v_5, v_8\rangle,\ 
\langle v_5, v_6, v_8\rangle,\ 
\langle v_3, v_7, v_8\rangle\mbox{ and }
\langle v_6, v_7, v_8\rangle, 
\]
because $v_3+v_6=0$. Actually, $\Delta$ must contain all of them. 
This means that $\Delta=\Sigma$ holds. 
We finish the proof. 
\end{proof}

\begin{rem}\label{p-rem5.3} 
The condition in Lemma \ref{p-lem5.1} is only a {\em sufficient} condition. In fact, 
for example, we see 
$\Delta=\Sigma$ for $(a,b,c,d)=(2,3,5,7)$, 
though the condition in Lemma \ref{p-lem5.1} is not satisfied. 
Moreover, for many other cases, one can confirm that 
$\Sigma$ is the unique non-singular complete fan whose 
one-dimensional cones are $\<v_1\>, \ldots, \<v_8\>$ in 
[8-13$''$] as the proof of Lemma \ref{p-lem5.1}. 
\end{rem} 

For example, $(a,b,c,d)=(2,7,4,2)$ satisfies the conditions in Lemma \ref{p-lem5.1} 
(see Remark \ref{p-rem5.2}, too). Thus, we obtain the following. 

\begin{cor}\label{p-cor5.4}
Let $Z''_{13}(2, 7, 4, 2)\dashrightarrow X'$ be a toric birational 
map which is an isomorphism in codimension one. 
Assume that $X'$ is projective. 
Then $X'$ is always singular. 
\end{cor}

\begin{proof}
We have already proved that $Z''_{13}(2, 7, 4, 2)$ is non-projective 
in Section \ref{p-sec4}. 
Hence this follows from Lemma \ref{p-lem5.1}. 
\end{proof}

Note that $Z''_{13}(2, 7, 4, 2)$ is a desired smooth non-projective 
complete toric threefold with Picard number five in Theorem \ref{p-thm1.7}. 
Corollary \ref{p-cor5.4} says that 
we can not make $X'$ smooth even if $X$ is so 
in Theorem \ref{p-thm1.3}. 

\begin{rem}[see Remark \ref{p-rem2.4}]\label{p-rem5.5}
Since $Z''_{13}(2, 7, 4, 2)$ is a smooth toric threefold, 
it has no flipping contractions (see \cite{fstu}). Moreover, 
by Lemma \ref{p-lem5.1}, it has 
no flopping contractions since every three-dimensional 
toric flop preserves smoothness. 
\end{rem}


\begin{thebibliography}{FSTU}

\bibitem[Ba1]{batyrev1} 
V.~Batyrev, 
On the classification of smooth projective toric varieties, 
Tohoku Math. J. {\textbf{43}} (1991), no. 4, 569--585.

\bibitem[Ba2]{batyrev2} 
V.~Batyrev, 
On the classification of toric Fano 4-folds, 
Algebraic Geometry, 9, J. Math. Sci. (New York) 
{\textbf{94}} (1999), 1021--1050.


\bibitem[Bo]{bonavero}
L.~Bonavero, 
Sur des vari\'et\'es toriques non projectives, 
Bull. Soc. Math. France \textbf{128} (2000), no. 3, 407--431.

\bibitem[CLS]{cls}
D.~A.~Cox, J.~B.~Little, H.~K.~Schenck, 
{\em{Toric varieties}}, 
Graduate Studies in Mathematics, 
\textbf{124}. American Mathematical Society, Providence, RI, 2011.

\bibitem[F1]{fujino}
O.~Fujino, 
Notes on toric varieties from Mori theoretic viewpoint, 
Tohoku Math. J. (2) \textbf{55} (2003), no. 4, 551--564.

\bibitem[F2]{fujino-km-cone}
O.~Fujino, 
On the Kleiman--Mori cone, 
Proc. Japan Acad. Ser. A Math. Sci. \textbf{81} (2005), no. 5, 80--84.

\bibitem[F3]{fujino-equivariant} 
O.~Fujino, 
Equivariant completions of toric contraction morphisms, With an appendix 
by Fujino and Hiroshi Sato, 
Tohoku Math. J. (2) \textbf{58} (2006), no. 3, 303--321. 

\bibitem[F4]{fujino-foundations} 
O.~Fujino, {\em{Foundations of the minimal model program}}, 
MSJ Memoirs, \textbf{35}. Mathematical Society of Japan, Tokyo, 2017.

\bibitem[F5]{fujino-classC}
O.~Fujino, 
Minimal model theory for log surfaces in Fujiki's class $\mathcal C$, 
Nagoya Math. J. \textbf{244} (2021), 256--282. 

\bibitem[FP]{fujino-payne}
O.~Fujino, S.~Payne, 
Smooth complete toric threefolds with no nontrivial nef line 
bundles, 
Proc. Japan Acad. Ser. A Math. Sci. \textbf{81} (2005), no. 10, 174--179. 

\bibitem[FS]{fujino-sato} 
O.~Fujino, H.~Sato, 
Introduction to the toric Mori theory, 
Michigan Math. J. \textbf{52} (2004), no. 3, 649--665.

\bibitem[FSTU]{fstu} 
O.~Fujino, H.~Sato, Y.~Takano, H.~Uehara, 
Three-dimensional terminal toric flips, 
Cent. Eur. J. Math. \textbf{7} (2009), no. 1, 46--53. 

\bibitem[Fl]{fulton}
W.~Fulton, {\em{Introduction to toric varieties}}, 
Annals of Mathematics Studies, \textbf{131}. The William 
H.~Roever Lectures in Geometry. Princeton 
University Press, Princeton, NJ, 1993.

\bibitem[KS]{ks} 
P.~Kleinschmidt, B.~Sturmfels, 
Smooth toric varieties with small Picard number are projective, 
Topology \textbf{30} (1991), no. 2, 289--299.

\bibitem[K]{kollar}
J.~Koll\'ar, 
Flips, flops, minimal models, etc, 
Surveys in differential geometry (Cambridge, MA, 1990), 113--199, 
Lehigh Univ., Bethlehem, PA, 1991.

\bibitem[M]{matsuki} 
K.~Matsuki, {\em{Introduction to the Mori program}}, 
Universitext. Springer-Verlag, New York, 2002.

\bibitem[N]{nagaya} 
O.~Nagaya, Classification of $3$-dimensional 
complete non-singular torus embeddings, Master's thesis, 
Nagoya Univ. 1976. 

\bibitem[O1]{oda1}
T.~Oda, {\em{Torus embeddings and applications}}, 
Based on joint work with Katsuya Miyake. Tata 
Institute of Fundamental Research Lectures on 
Mathematics and Physics, \textbf{57}. Tata Institute 
of Fundamental Research, Bombay; Springer-Verlag, 
Berlin-New York, 1978.

\bibitem[O2]{oda2}
T.~Oda, {\em{Convex bodies and algebraic geometry. An introduction to 
the theory of toric varieties}}, 
Translated from the Japanese. Ergebnisse 
der Mathematik und ihrer Grenzgebiete (3) [Results 
in Mathematics and Related Areas (3)], \textbf{15}. Springer-Verlag, 
Berlin, 1988.

\bibitem[Re]{reid}
M.~Reid, 
Decomposition of toric morphisms, 
{\em{Arithmetic and geometry, Vol. II}}, 395--418, 
Progr. Math., \textbf{36}, Birkh\"auser Boston, Boston, MA, 1983. 

\bibitem[RT]{rt}
M.~Rossi, L.~Terracini, 
A $\mathbb Q$-factorial complete toric variety with Picard number $2$ 
is projective, J. Pure Appl. Algebra \textbf{222} (2018), no. 9, 2648--2656.

\bibitem[Sa1]{sato1}
H.~Sato, 
Toward the classification of higher-dimensional 
toric Fano varieties, Tohoku Math. J. {\textbf{52}} (2000), 
no. 3, 383--413. 


\bibitem[Sa2]{sato2}
H.~Sato, Combinatorial descriptions 
of toric extremal contractions, 
Nagoya Math. J. \textbf{180} (2005), 111--120.

\bibitem[Su]{sumihiro} 
H.~Sumihiro, 
Equivariant completion, 
J. Math. Kyoto Univ. \textbf{14} (1974), 1--28.

\bibitem[T]{tanaka}
H.~Tanaka, 
Kawamata--Viehweg vanishing for toric varieties, 
Pure Appl. Math. Q. \textbf{21} (2025). 
\end{thebibliography}
\end{document}